\documentclass[a4paper,10pt]{article}
\usepackage{authblk}
\usepackage{amsthm}
\usepackage{booktabs}
\usepackage{comment}
\usepackage{nicefrac}
\usepackage{multirow}
\usepackage{breakurl}
\usepackage{cancel}
\usepackage[hmargin={28mm,28mm},vmargin={30mm,35mm}]{geometry}
\usepackage[ocgcolorlinks,allcolors={blue}]{hyperref}
\usepackage[latin1]{inputenc}
\usepackage{newtxtext,newtxmath}
\usepackage{nicefrac}
\usepackage{subcaption}
\usepackage[compact]{titlesec}
\usepackage{xcolor}

\newcommand{\email}[1]{\href{mailto:#1}{#1}}

\newtheorem{theorem}{Theorem}
\newtheorem{proposition}[theorem]{Proposition}
\newtheorem{lemma}[theorem]{Lemma}

\newtheorem{problem}{Problem}

\theoremstyle{remark}
\newtheorem{remark}[theorem]{Remark}
\theoremstyle{definition}
\newtheorem{assumption}{Assumption}

\newcommand{\bvec}[1]{\boldsymbol{#1}}
\newcommand{\uvec}[1]{\underline{\bvec{#1}}}
\newcommand{\btens}[1]{\boldsymbol{#1}}

\DeclareMathOperator{\tr}{tr}

\newcommand{\GRAD}{\bvec{\nabla}}
\newcommand{\GRADs}{\GRAD_\symm}
\newcommand{\GRADss}{\GRAD_{\rm ss}}
\newcommand{\DIV}{\bvec{\nabla}{\cdot}}
\newcommand{\CURL}{\bvec{\nabla}\times}
\newcommand{\LAPL}{\bvec{\Delta}}
\newcommand{\GRADh}{\GRAD_h}
\newcommand{\GRADsh}{\GRAD_{\symm,h}}

\newcommand{\ud}{\,\mathrm{d}}

\newcommand{\term}{\mathfrak{T}}

\newcommand{\st}{\; : \;}
\newcommand{\Id}[1][d]{\bvec{I}_{#1}}

\newcommand{\norm}[2][]{\|#2\|_{#1}}
\newcommand{\seminorm}[2][]{|#2|_{#1}}

\newcommand{\abs}[1]{\mathsf{#1}}

\newcommand{\symm}{{\rm s}}

\newcommand{\diff}{\btens{K}}

\newcommand{\Real}{\mathbb{R}}

\newcommand{\Integer}{\mathbb{Z}}

\newcommand{\Poly}[2][]{\mathbb{P}_{#1}^{#2}}
\newcommand{\Polys}[1]{\btens{\mathbb{P}}_{\symm}^{#1}}

\newcommand{\Mh}[1][h]{\mathcal{M}_{#1}}
\newcommand{\Th}[1][h]{\mathcal{T}_{#1}}
\newcommand{\Fh}[1][h]{\mathcal{F}_{#1}}
\newcommand{\Fhi}{\Fh^{{\rm i}}}
\newcommand{\Fhb}{\Fh^{{\rm b}}}
\newcommand{\normal}{\bvec{n}}

\newcommand{\tF}{t_{\rm F}}
\newcommand{\dt}{d_t}

\newcommand{\lproj}[2][h]{\pi_{#1}^{#2}}
\newcommand{\vlproj}[2][h]{\bvec{\pi}_{#1}^{#2}}

\newcommand{\jump}[1]{[#1]_F}
\newcommand{\wavg}[1]{\{#1\}_{\omega,F}}

\newcommand{\ET}[1][k]{\btens{\mathrm{E}}_T^{#1}}
\newcommand{\rT}[1][k+1]{\bvec{\mathrm{r}}_T^{#1}}
\newcommand{\Eh}[1][k]{\btens{\mathrm{E}}_h^{#1}}
\newcommand{\Dh}[1][k]{\mathrm{D}_h^{#1}}
\newcommand{\rh}[1][k+1]{\bvec{\mathrm{r}}_h^{#1}}
\newcommand{\pT}[1][k+1]{\mathrm{p}_T^{#1}}
\newcommand{\ph}[1][k+1]{\mathrm{p}_h^{#1}}

\newcommand{\strain}{\btens{\varepsilon}}

\newcommand{\hho}{\mathrm{hho}}
\newcommand{\dg}{\mathrm{dg}}

\begin{document}

\title{An abstract analysis framework for monolithic discretisations of poroelasticity with application to Hybrid High-Order methods}
\author[1]{Lorenzo Botti\footnote{\email{lorenzo.botti@unibg.it}}}
\author[2]{Michele Botti\footnote{\email{michele.botti@polimi.it}}}
\author[3]{Daniele A. Di Pietro \footnote{\email{daniele.di-pietro@umontpellier.fr}}}
\affil[1]{Department of Engineering and Applied Sciences, University of Bergamo, Italy}
\affil[2]{MOX, Department of Mathematics, Politecnico di Milano, Milano, Italy}
\affil[3]{IMAG, Univ Montpellier, CNRS, Montpellier, France}

\maketitle

\begin{abstract}
  In this work, we introduce a novel abstract framework for the stability and convergence analysis of fully coupled discretisations of the poroelasticity problem and apply it to the analysis of Hybrid High-Order (HHO) schemes.
  A relevant feature of the proposed framework is that it rests on mild time regularity assumptions that can be derived from an appropriate weak formulation of the continuous problem.
  To the best of our knowledge, these regularity results for the Biot problem are new.
  A novel family of HHO discretisation schemes is proposed and analysed, and their performance numerically evaluated.
\end{abstract}


\section{Introduction}

The numerical modeling of poroelastic media is relevant in several applications in geosciences, including subsidence due to fluid withdrawal, reservoir impoundment, tensile failure induced by pressurization of a borehole, waste disposal, earthquake triggering due to pressure-induced faults slip, injection-production cycles in geothermal fields, and carbon dioxide storage in saline aquifers.
The interest in the coupled diffusion-deformation mechanisms was initially motivated by the problem of consolidation, namely the progressive settlement of a soil due to fluid extraction.
The earliest theory modeling the effects of the pore fluid on the deformation of the soil was developed in the pioneering work of Terzaghi \cite{Terzaghi:23}, who proposed a model for consolidation accounting for the fluid-to-solid interaction only. In this case, the problem can be decoupled and solved in two stages. This kind of theory can successfully model some of the poroelastic processes in the case of highly compressible fluids such as air. However, when dealing with slightly compressible (or incompressible) fluids, the solid-to-fluid interaction cannot be neglected, since the changes in the stress can significantly influence the pore pressure. The first detailed mathematical theory of poroelasticity incorporating two way interactions was formulated by Biot \cite{Biot:41}.
The model proposed by Biot was subsequently re-derived via homogenization \cite{Auriault.Sanchez-Palencia:77,Burridge.Keller:81} and mixture theory \cite{Bowen:80,Bowen:82}, which placed it on a rigorous basis.

Several space discretisation methods for the Biot problem have been considered in the literature.
Finite element discretisations are discussed, e.g., in the monograph \cite{Lewis.Schrefler:98}; cf. also references therein.
A finite volume discretisation for the three-dimensional Biot problem with discontinuous physical coefficients is considered in \cite{Naumovich:06}.
In \cite{Phillips.Wheeler:07,Phillips.Wheeler:07*1}, an algorithm that models displacements with continuous elements and the flow with a mixed method is proposed.
In \cite{Phillips.Wheeler:08}, the same authors study a different method, where displacements are approximated using discontinuous Galerkin (DG) methods.
The coupling of a multipoint flux discretisation of the flow with a DG discretisation of the mechanical term is studied in \cite{Wheeler.Xue.ea:14}.
More recently, stable cell centered finite volume discretisations have been considered in \cite{Nordbotten:16}.
Concerning the Hybrid High-Order (HHO) literature, we can cite the original method of \cite{Boffi.Botti.ea:16}, relying on a discretisation of the displacements in the HHO space of degree $k\ge 1$ and of the pore pressure in the space of broken polynomials of total degree $\le k$, as well as its extension to nonlinear strain--stress laws considered in \cite{Botti.Di-Pietro.ea:19} (based in turn on the method of \cite{Botti.Di-Pietro.ea:17} for nonlinear elasticity).
We also cite here the recent work \cite{Botti.Di-Pietro.ea:20} on the extension to random coefficients.

The contribution of the present work is threefold:
first, we prove novel regularity estimates for the solution of a weak formulation of the continuous Biot problem;
second, in the spirit of \cite{Di-Pietro.Droniou:18}, we introduce a novel abstract framework for the stability and convergence analysis of fully coupled discretisation schemes that enables one to exploit the aforementioned regularity;
third, we apply the abstract framework to the analysis of both existing and novel schemes based on HHO discretisations of the mechanical term.
Specifically, we introduce a novel family of schemes where HHO methods are used to discretise both the mechanical and flow terms.
These HHO-HHO discretisations display some relevant advantages in addition to the properties of the original HHO-DG scheme of \cite{Boffi.Botti.ea:16}:
they enable the use of piecewise constant approximations of the displacements and the pore pressure, leading to schemes that lend themselves to large three-dimensional industrial simulations;
when polynomials of total degree $\le k$ are used as discrete unknowns, convergence in $h^{k+1}$ (with $h$ denoting, as usual, the meshsize) can be proved for the energy norm of the error without requiring elliptic regularity (that is, for non convex domains and/or heterogeneous permeabilities);
the implementation can benefit from static condensation in the spirit of \cite[Section 6]{Di-Pietro.Ern.ea:16*1} to locally eliminate a large subset of both displacement and pressure unknowns.
We notice, in passing, that the error estimates derived here for the HHO-DG method are different from the ones originally proved in \cite{Boffi.Botti.ea:16} and hold under milder time regularity requirements.

The rest of this paper is organised as follows.
In Section \ref{sec:continuous.setting} we establish the continuous setting, state a weak formulation of the Biot problem, and derive novel regularity results for its solution.
Section \ref{sec:abstract.framework} contains the abstract analysis framework, including the assumptions on the abstract scheme as well as its stability and convergence analysis.
In Section \ref{sec:discrete.setting} we introduce the discrete setting, while in Section \ref{sec:hho} we formulate the HHO schemes. Their convergence analysis based on the abstract framework of Section \ref{sec:abstract.framework} is carried out in Section \ref{sec:error.analysis}.
Finally, some numerical examples are presented in Section \ref{sec:numerical.examples} to corroborate the theoretical results.


\section{Continuous setting}\label{sec:continuous.setting}

In this section we state the Biot problem and its weak formulation, and discuss the regularity of the solution.

\subsection{Continuous problem}

Let $\Omega\subset\Real^d$, $d\in\{2,3\}$, denote a bounded connected polyhedral domain with boundary $\partial\Omega$ and outward normal $\normal$.
For a given observation time $\tF>0$, volumetric load $\bvec{f}:\Omega\times(0,\tF)\to\Real^d$, and
fluid source $g:\Omega\times(0,\tF)\to\Real$,
we consider the poroelasticity problem that consists in finding a 
vector-valued displacement field $\bvec{u}:\Omega\times\lbrack 0,\tF)\to\Real^d$ and a 
scalar-valued pore pressure field $p:\Omega\times\lbrack 0,\tF)\to\Real$ such that
\begin{subequations}\label{eq:biot:strong}
  \begin{alignat}{2}
    \label{eq:biot:strong:mechanics}
    -\DIV\btens{\sigma}(\GRADs\bvec{u}) + C_{\rm bw}\GRAD p &= \bvec{f} &\qquad&\text{in $\Omega\times (0,\tF)$},
    \\
    \label{eq:biot:strong:flow}
    C_0\dt p + C_{\rm bw}\dt(\DIV\bvec{u}) - \DIV(\diff\GRAD p) &= g &\qquad&\text{in $\Omega\times (0,\tF)$}.
  \end{alignat}
  The quantities appearing in the above equations are defined as follows.
  Let $\Real^{d\times d}_\symm$ denote the set of symmetric $d\times d$ real-valued matrices.
  In \eqref{eq:biot:strong:mechanics}, $\GRADs$ denotes the symmetric part of the gradient operator acting on vector-valued fields, $\btens{\sigma}:\Real^{d\times d}_\symm\to\Real^{d\times d}_\symm$ is the linear strain-stress law such that, for a given uniformly elliptic fourth-order tensor-valued function $\btens{C}:\Omega\to\Real^{d^4}$,
  \begin{equation}\label{eq:C.lame}
    \btens{\sigma}(\btens{\tau}) = \btens{C}\btens{\tau}\qquad\forall\btens{\tau}\in\Real^{d\times d}_\symm,
  \end{equation}
  and $C_{\rm bw}>0$ is the Biot--Willis coefficient.
  For homogeneous isotropic materials, the tensor $\btens{C}$ can be expressed in terms of the Lam\'e coefficients $\mu>0$ and $\lambda$ such that $2\mu+d\lambda\ge 0$:
  For all $\btens{\tau}\in\Real_\symm^{d\times d}$,
  \begin{equation}\label{eq:def_tensC}
    \btens{C}\btens{\tau} = 2\mu\ \btens{\tau} + \lambda\ {\rm tr}(\btens{\tau})\btens{I}_d 
    = 2\mu\ {\rm \textbf{dev}}(\btens{\tau}) + \frac{2\mu+d\lambda}d\ {\rm tr}(\btens{\tau})\btens{I}_d,
  \end{equation}
  where $\btens{I}_d$ denotes the identity matrix of $\Real^{d\times d}$, 
  ${\rm tr}(\btens{\tau}) \coloneq \sum_i^d \tau_{ii}$, and 
  ${\rm \textbf{dev}}(\btens{\tau})\coloneq \btens{\tau} - d^{-1}{\rm tr}(\btens{\tau})\Id$.
  In \eqref{eq:biot:strong:flow}, $\dt$ denotes the time derivative, $C_0\ge 0$ is the constrained specific storage 
  coefficient, while $\diff:\Omega\to\Real^{d\times d}_\symm$ is the uniformly elliptic hydraulic mobility tensor field 
  which, for strictly positive real numbers $0<\underline{K}\le\overline{K}$, satisfies 
  \begin{equation}\label{eq:ass_diff}
    \text{
      $\underline{K}\seminorm{\bvec{\xi}}^2\le\diff(\bvec{x})\bvec{\xi}\cdot\bvec{\xi}\le\overline{K}\seminorm{\bvec{\xi}}^2$
      for almost every (a.e.) $\bvec{x}\in\Omega$ and all $\bvec{\xi}\in\Real^d$.
    }
  \end{equation}
  In the poroelasticity theory \cite{Coussy:04}, the medium is modeled as a continuous superposition of solid and fluid phases.
  Following \cite{Terzaghi:43}, the mechanical equilibrium equation~\eqref{eq:biot:strong:mechanics} is based on the decomposition  of the total stress tensor into a mechanical contribution and a pore pressure contribution.
  The mass conservation equation~\eqref{eq:biot:strong:flow} is derived for fully saturated porous media assuming Darcean flow. The first two terms of this equation quantify the variation of fluid content in the pores. The dimensionless coupling coefficient $C_{\rm bw}$ expresses the amount of fluid that can be forced into the medium by a variation of the pore volume at a constant fluid pressure, while $C_0$ measures the amount of fluid that can be forced into the medium by pressure increments due to compressibility of the structure. The case of a solid matrix with incompressible grains corresponds to the limit value $C_0=0$.

  To close the problem, we enforce homogeneous boundary conditions corresponding to a clamped, impermeable boundary:
  \begin{align}
    \label{eq:biot:strong:bc.u}
    \bvec{u} &= \bvec{0} \qquad\text{on $\partial\Omega\times(0,\tF)$},
    \\
    \label{eq:biot:strong:bc.p}
    (\diff\GRAD p)\cdot\normal &= 0 \qquad\text{on $\partial\Omega\times(0,\tF)$},
  \end{align}
  as well as the following initial condition which prescribes the initial fluid content:
  \begin{equation}\label{eq:biot:strong:initial}
    C_0 p(\cdot,0)+C_{\rm bw}\DIV\bvec{u}(\cdot,0)=\phi^0(\cdot). 
  \end{equation}
  In the case $C_0=0$, we also need the following compatibility conditions on $g$ and $\phi^0$ and zero mean value constraint on $p$:
  \begin{equation}\label{eq:compatibility}
    \int_\Omega \phi^0(\bvec{x})\ud\bvec{x} = 0,\qquad
    \int_\Omega g(\bvec{x},t)\ud\bvec{x} = 0\quad\forall t\in(0,\tF),\qquad
    \int_\Omega p(\bvec{x},t)\ud\bvec{x} = 0\quad\forall t\in\lbrack0,\tF).
  \end{equation}
\end{subequations}
For the sake of simplicity, throughout the rest of the paper we make the usual assumption
\[
C_{\rm bw}=1.
\]

\subsection{Weak formulation}

We next discuss a weak formulation of problem \eqref{eq:biot:strong}.
For any measured set $X$, integer $m\in\Integer$, and vector space $V$, we denote by $H^m(X;V)$ the usual Sobolev space of $V$-valued functions that have weak partial derivatives of order up to $m$ in $L^2(X;V)$, with the convention that $H^0(X;V)\coloneq L^2(X;V)$.
$C^m(X;V)$ and $C_{\rm c}^\infty(X;V)$ denote, respectively, the usual spaces of $m$-times continuously differentiable $V$-valued functions and infinitely continuously differentiable $V$-valued functions with compact support on $X$.
For the sake of conciseness, we adopt the convention that the codomain is omitted when $V=\Real$, and we simply write $H^m(X)$, $L^2(X)$, $C^m(X)$, and $C_{\rm c}^m(X)$.
Denoting by $(\cdot,\cdot)_V$ and $\norm[V]{{\cdot}}$ the inner product and induced norm of $V$, the space $C^m([0,\tF];V)$ is a Banach space when equipped with the norm
$$
\norm[{C^{m}([0,\tF];V)}]{\varphi}\coloneq\max_{0\le i\le m}\max_{t\in[0,\tF]}\norm[V]{\dt^{(i)}\varphi(t)},
$$
where $\dt^{(i)}$ denotes the $i$th time derivative.
The Hilbert space $H^m(0,\tF;V)$ is equipped with the norm $\norm[H^m(0,\tF;V)]{{\cdot}}$ induced by the scalar product
$$
(\varphi,\psi)_{H^{m}(0,\tF;V)}
\coloneq\sum_{j=0}^m \int_0^{\tF} (\dt^{(j)}\varphi(t),\dt^{(j)}\psi(t))_V \ud t
\qquad \forall\varphi,\psi \in H^{m}(V).
$$

At a given time $t$, the natural functional spaces for the displacement $\bvec{u}(t)$ and pore pressure $p(t)$ taking into account the boundary condition \eqref{eq:biot:strong:bc.u} and, if $C_0=0$, the zero-mean value constraint \eqref{eq:compatibility} are, respectively, 
$$
\bvec{U}\coloneq H^1_0(\Omega)^d
\qquad\text{and}\qquad
P\coloneq\begin{cases}
H^1(\Omega) & \text{if $C_0>0$,} \\ 
H^1(\Omega)\cap L^2_0(\Omega)  & \text{if $C_0=0$,}
\end{cases}
$$
where $H^1_0(\Omega)$ is spanned by functions in $H^1(\Omega)$ that vanish on $\partial\Omega$ in the sense of traces and $L^2_0(\Omega)$ by functions in $L^2(\Omega)$ with zero mean value over $\Omega$.

Assume $\bvec{f}\in L^2(0,\tF; L^2(\Omega)^d)$ and $g\in L^2(0,\tF;L^2(\Omega))$ verifying \eqref{eq:compatibility} if $C_0=0$, and $\phi^0\in L^2(\Omega)$.
We consider the following weak formulation of problem~\eqref{eq:biot:strong}:
Find $\bvec{u}\in L^2(0,\tF;\bvec{U})$ and $p\in L^2(0,\tF;P)$ such that, for all $\bvec{v}\in\bvec{U}$, all $q\in P$, all $r\in L^2(\Omega)$, and all $\varphi,\psi\in C_{\rm c}^\infty((0,\tF))$,
\begin{subequations}
  \label{eq:weak_form}
  \begin{align}
    \label{eq:weak_form.mech}
    \int_0^{\tF}
    \left[
      a(\bvec{u}(t),\bvec{v}) + b(\bvec{v},p(t))
      \right]
    \varphi(t)\ud t
    &=\int_0^{\tF}\hspace{-1mm}(\bvec{f}(t),\bvec{v})_{L^2(\Omega)^d}\,\varphi(t) \ud t,
    \\
    \label{eq:weak_form.fluid}    
    \int_0^{\tF}\left[
      \left(
      -(C_0 p(t),q)_{L^2(\Omega)} + b(\bvec{u}(t),q)
      \right)\dt\psi(t)      
      + c(p(t),q)\psi(t)
      \right]\ud t 
    &=\int_0^{\tF}\hspace{-1mm}(g(t),q)_{L^2(\Omega)}\,\psi(t) \ud t, 
    \\
    \label{eq:weak_form.initial}
    (C_0 p(0) + \DIV\bvec{u}(0),r)_{L^2(\Omega)} &= (\phi^0,r)_{L^2(\Omega)},
  \end{align}
\end{subequations}
with bilinear forms $a:\bvec{U}\times\bvec{U}\to\Real$, $b:\bvec{U}\times P\to\Real$, and $c:P\times P\to\Real$ such that
\begin{equation}\label{eq:a.b.c}
  a(\bvec{w},\bvec{v}) \coloneq (\btens{C}\GRADs\bvec{w},\GRADs\bvec{v})_{L^2(\Omega)^{d\times d}},\qquad
  b(\bvec{v},q) \coloneq -(\DIV\bvec{v},q)_{L^2(\Omega)},\qquad
  c(r,q) \coloneq (\diff\GRAD r, \GRAD q)_{L^2(\Omega)^d}.
\end{equation}

\subsection{Regularity}
In this section, we investigate the regularity of the solution to problem \eqref{eq:weak_form}. 
Specifically, under mild assumptions on the problem data and initial pressure, we prove that $\bvec{u}\in H^1(0,\tF;\bvec{U})$ and $p\in C^0([0,\tF];P)\cap H^1(0,\tF;L^2(\Omega))$. For the sake of brevity, in this section we will often use the notation $a\lesssim b$ for the inequality $a\le Cb$ with generic constant $C>0$ independent of $\mu$, $\lambda$, $C_0$, and $\diff$, but possibly depending on $d$ and $\Omega$.

\subsubsection{Preliminary results}

We start by recalling two fundamental results that will be needed in the discussion.
The following generalised Korn inequality is proved in \cite[Theorem 2]{Reshetnyak:70} and \cite[Theorem 2.3]{Schirra:12}:
  There exists a positive constant $C_{\rm K}$, depending on $\Omega$ and $d$, such that
  \begin{equation}\label{eq:korn_gen} 
  \norm[H^1(\Omega)^d]{\bvec{v}} \le C_{\rm K} \norm[L^2(\Omega)^{d\times d}]{\mathbf{dev}(\GRADs\bvec{v})}
  \qquad\forall\bvec{v}\in \bvec{U}.
  \end{equation}
As a consequence of \eqref{eq:korn_gen} and of the definition of $\btens{C}$ given in \eqref{eq:def_tensC}, one has
\begin{equation}\label{eq:cv_bnd_StressStrain}
  2\mu\norm[H^1(\Omega)^d]{\bvec{v}}^2
  \le C_{\rm K}^2 a(\bvec{v},\bvec{v})
  \le (2\mu + d\lambda^+) C_{\rm K}^2\norm[H^1(\Omega)^d]{\bvec{v}}^2,
\end{equation} 
with $\lambda^+$ denoting the positive part of $\lambda$.
The second, classical result is the surjectivity of $\DIV:\bvec{U}\to L^2_0(\Omega)$ (see, e.g., \cite{Ladyzhenskaya:69} and \cite[Corollary 2.4]{Girault.Raviart:86}):
There exists $\tilde{\beta}>0$ only depending on $\Omega$ and $d$ such that
\begin{equation}\label{eq:inf-sup:continuous}
  \tilde{\beta}\norm[L^2(\Omega)]{q}\le 
  \sup_{\bvec{v}\in\bvec{U}, \norm[H^1(\Omega)^d]{\bvec{v}}=1} b(\bvec{v},q)
  \qquad\forall q\in L^2_0(\Omega).
\end{equation}

\subsubsection{Regularity estimate}
Assume $\bvec{f}\in H^1(0,\tF;L^2(\Omega)^d)$ and $\phi^0\in H^1(\Omega)$.
This additional regularity of the loading term $\bvec{f}$ allows to prescribe the mechanical equilibrium at the initial time $t=0$.
Specifically, observing that, for all $\bvec{v}\in \bvec{U}$, the function $t\mapsto(\bvec{f}(t),\bvec{v})_{L^2(\Omega)}$ belongs to $H^1{(0,\tF)\subset C^0([0,\tF])}$, we consider the initial solution $(\bvec{u}(0),p(0))\in\bvec{U}\times P$ solving the weak problem given by \eqref{eq:weak_form.initial} along with the mechanical equilibrium equation
\begin{equation}\label{eq:initial_mech}
  a(\bvec{u}(0),\bvec{v})
  + b(\bvec{v},p(0))
  = (\bvec{f}(0),\bvec{v})_{L^2(\Omega)^d}.
\end{equation}
We will additionally need the following assumption on the initial pressure:
\begin{equation}\label{eq:reg_p_initial}
\norm[L^2(\Omega)^d]{\GRAD p(0)} \lesssim \mathcal{C}
\left(\norm[L^2(\Omega)^d]{\bvec{f}(0)} + \norm[H^1(\Omega)]{\phi^0}\right),
\end{equation}
where the constant $\mathcal{C}$ possibly depends on the problem coefficients $\mu$, $\lambda$, and $C_0$.
In Section \ref{sec:regularity:p(0)}, we will discuss some sufficient assumptions on the problem data under which \eqref{eq:reg_p_initial} holds.
\begin{theorem}[Regularity estimate]
\label{thm:reg_est}
  Let $(\bvec{u},p)\in L^2(0,\tF;\bvec{U}\times P)$ solve problem \eqref{eq:weak_form} with 
  $\bvec{f}\in H^1(0,\tF;L^2(\Omega)^d)$ and $\phi^0\in H^1(\Omega)$. Assume, additionally, that \eqref{eq:reg_p_initial} 
  holds. Then,
  \[
  \text{
    $\bvec{u}\in H^1(0,\tF;\bvec{U})$ and $p\in C^0([0,\tF];P)\cap H^1(0,\tF;L^2(\Omega))$.
  }
  \]
\end{theorem}
\begin{proof}
  Owing to $\bvec{f}\in H^1(0,\tF;L^2(\Omega)^d)$, we can take $\varphi=-\dt\psi\in C_{\rm c}^\infty((0,\tF))$ in \eqref{eq:weak_form.mech}, then use an integration by parts in time in both \eqref{eq:weak_form.mech} and \eqref{eq:weak_form.fluid} to infer
  \begin{subequations}
    \begin{alignat}{2}
      \label{eq:reg_mech}
      a(\dt\bvec{u},\bvec{v}) + b(\bvec{v},\dt p)
      &= (\dt\bvec{f},\bvec{v})_{L^2(\Omega)^d}
      &\qquad&\text{$\forall\bvec{v}\in\bvec{U}$ in $L^2(0,\tF)$},
      \\
      \label{eq:reg_fluid}
      C_0 (\dt p,q)_{L^2(\Omega)}
      - b(\dt\bvec{u},q)
      + c(p,q)
      &= (g,q)_{L^2(\Omega)}
      &\qquad&\text{$\forall q\in P$ in $L^2(0,\tF)$}. 
    \end{alignat}
  \end{subequations}
  Let now $s\in(0,\tF]$. Taking $\bvec{v}=\dt \bvec{u}$ in \eqref{eq:reg_mech} and $q =\dt p$ in \eqref{eq:reg_fluid}, 
  summing the resulting equations, and integrating over $(0,s)$, yields
    \begin{multline*}
      \int_0^{s}\left[
      a(\dt\bvec{u}(t),\dt\bvec{u}(t))
      + C_0\norm[L^2(\Omega)]{\dt p(t)}^2
      + c(p(t),\dt p(t))
      \right]
      \ud t =
      \\
      \int_0^{s}\left[
        (\dt\bvec{f}(t),\dt \bvec{u}(t))_{L^2(\Omega)^d} + (g(t),\dt p(t))_{L^2(\Omega)}
        \right]\ud t.
    \end{multline*}
    Using the coercivity of the bilinear form $a$ expressed by the leftmost inequality in \eqref{eq:cv_bnd_StressStrain}, 
    the identity $c(\varphi,\dt\varphi)=\frac12\dt c(\varphi,\varphi)$, and recalling the uniform ellipticity 
    assumption \eqref{eq:ass_diff} on the hydraulic mobility tensor, it follows from the previous relation that
    \begin{multline}\label{eq:basic_reg_est1}
      \int_0^{s}\left[
        \frac{2\mu}{C_{\rm K}^2}\norm[H^1(\Omega)^d]{\dt\bvec{u}(t)}^2 + C_0\norm[L^2(\Omega)]{\dt p(t)}^2
        \right]\ud t
      +\frac{\underline{K}}2 \norm[{L^2(\Omega)^d}]{\GRAD p(s)}^2 -
      \frac{\overline{K}}2 \norm[{L^2(\Omega)^d}]{\GRAD p(0)}^2
      \\
      \le \int_0^{s}\left[
        (\dt\bvec{f}(t),\dt \bvec{u}(t))_{L^2(\Omega)^d} 
        + (g(t),\dt (p-\lproj[\Omega]{0}p)(t))_{L^2(\Omega)} 
        + (\lproj[\Omega]{0}g(t),\dt p(t))_{L^2(\Omega)}
        \right]\ud t,
    \end{multline}
    where the spatial mean-value operator $\lproj[\Omega]{0}:L^1(\Omega)\to\Real$ is defined such that, 
    for all $\varphi\in L^1(\Omega)$,
    \[
    \lproj[\Omega]{0}\varphi\coloneq\frac{1}{|\Omega|}\int_\Omega \varphi(\bvec{x})\ud\bvec{x}.
    \]
    (Notice that $\lproj[\Omega]{0}\varphi$ is identified with a constant function over $\Omega$ when needed.)

    Using the continuous inf-sup condition \eqref{eq:inf-sup:continuous}, equation \eqref{eq:reg_mech}, the continuity 
    of the bilinear form $a$ expressed by rightmost inequality in \eqref{eq:cv_bnd_StressStrain}, 
    and the Cauchy--Schwarz inequality, we obtain, for a.e. $t\in(0,s)$,
    \begin{equation}\label{eq:reg_est_inf-sup}
    \begin{aligned}
      \tilde{\beta}\norm[L^2(\Omega)]{\dt (p-\lproj[\Omega]{0}p)(t)} &\le
      \sup_{\bvec{v}\in\bvec{U}, \norm[H^1(\Omega)^d]{\bvec{v}}=1}\left[
      a(\dt\bvec{u}(t),\bvec{v})
      - (\dt\bvec{f}(t),\bvec{v})_{L^2(\Omega)^d}
      \right]
      \\
      &\le (2\mu+d\lambda^+)\norm[H^1(\Omega)^d]{\dt\bvec{u}(t)}
      +\norm[{L^2(\Omega)^d}]{\dt\bvec{f}(t)}.
    \end{aligned}
    \end{equation}
    Thus, using the Cauchy--Schwarz inequality, the previous estimate, and the Young inequality to bound 
    the right-hand side of \eqref{eq:basic_reg_est1}, leads to
    $$
    \begin{aligned}
      &\int_0^{s}\left[
      \frac{2\mu}{C_{\rm K}^2}\norm[H^1(\Omega)^d]{\dt\bvec{u}(t)}^2 + C_0\norm[L^2(\Omega)]{\dt p(t)}^2 \right]\ud t
      +\frac{\underline{K}}2 \norm[{L^2(\Omega)^d}]{\GRAD p(s)}^2 -
      \frac{\overline{K}}2 \norm[{L^2(\Omega)^d}]{\GRAD p(0)}^2
      \\
      &\quad\le \int_0^{s}
        \norm[H^1(\Omega)^d]{\dt\bvec{u}(t)} \left(
        \norm[L^2(\Omega)^d]{\dt\bvec{f}(t)} + \frac{2\mu+d\lambda^+}{\tilde{\beta}} \norm[L^2(\Omega)]{g(t)} \right)
        + \frac{\norm[L^2(\Omega)]{g(t)} \norm[L^2(\Omega)^d]{\dt\bvec{f}(t)}}{\tilde{\beta}}
      \\
      &\qquad\qquad\qquad\qquad\qquad\qquad\qquad\qquad\qquad\qquad\qquad\qquad\qquad\,
      +\norm[L^2(\Omega)]{\dt p(t)} \norm[L^2(\Omega)]{\lproj[\Omega]{0}g(t)} \ud t
      \\
      &\quad\le\int_0^{s} \left[
          \frac{C_{\rm K}^2}{\mu}\norm[L^2(\Omega)^d]{\dt\bvec{f}(t)}^2
          +\frac{(2\mu+d\lambda^+)^2(4 C_{\rm K}^4 + 1)}{8\mu C_{\rm K}^2 \tilde{\beta}^2}\norm[L^2(\Omega)]{g(t)}^2 
          + \frac1{2 C_0}\norm[L^2(\Omega)]{\lproj[\Omega]{0}g(t)}^2 \right]\ud t
      \\
      &\qquad\qquad\qquad\qquad\qquad\qquad\qquad\qquad\qquad\qquad
      +\int_0^{s}\frac{\mu}{C_{\rm K}^2}\norm[H^1(\Omega)^d]{\dt\bvec{u}(t)}^2 
      + \frac{C_0}2 \norm[L^2(\Omega)]{\dt p(t)}^2\ud t,
    \end{aligned}
    $$
    where we have adopted the convention that $C_0^{-1}\norm[L^2(\Omega)]{\lproj[\Omega]{0}g(t)}=0$ if $C_0=0$.
    Rearranging, multiplying by $2$, and hiding, in the right-hand side, a multiplicative costant that depends only on $\Omega$ and $d$ for the sake of legibility, 
    we get
    \begin{multline}\label{eq:basic_reg_est2}
      \int_0^{s}\left[
        2\mu\norm[H^1(\Omega)^d]{\dt\bvec{u}(t)}^2+C_0\norm[L^2(\Omega)]{\dt p(t)}^2
        \right]\ud t
      +\underline{K} \norm[{L^2(\Omega)^d}]{\GRAD p(s)}^2 
      -\overline{K}\norm[{L^2(\Omega)^d}]{\GRAD p(0)}^2
      \\
      \lesssim
      \mu^{-1}\norm[H^1(0,\tF;L^2(\Omega)^d)]{\bvec{f}}^2
      + (2\mu+d\lambda^+)^2\mu^{-1}\norm[L^2(0,\tF;L^2(\Omega))]{g}^2 
      + {C_0}^{-1}\norm[L^2(0,\tF;L^2(\Omega))]{\lproj[\Omega]{0}g}^2.
    \end{multline}
    Owing to \eqref{eq:reg_p_initial} along with the fact that 
    $\norm[L^2(\Omega)^d]{\bvec{f}(0)}\lesssim\norm[H^1(0,\tF;L^2(\Omega)3^d)]{\bvec{f}}$, we obtain
    $$
      \norm[{L^2(\Omega)^d}]{\GRAD p(0)}^2 \lesssim 
      \mathcal{C}^2\left(\norm[H^1(0,\tF;L^2(\Omega)^d)]{\bvec{f}}^2 + \norm[H^1(\Omega)]{\phi^0}^2\right).
    $$
    Hence, plugging the previous bound into \eqref{eq:basic_reg_est2}, it is inferred that
    \begin{multline}\label{eq:basic_reg_est3}
      \int_0^{s}\left[
      2\mu \norm[H^1(\Omega)^d]{\dt\bvec{u}(t)}^2+C_0\norm[L^2(\Omega)]{\dt p(t)}^2
      \right]\ud t
      +\underline{K} \norm[{L^2(\Omega)^d}]{\GRAD p(s)}^2 \lesssim
      \overline{K}\mathcal{C}^2 \norm[H^1(\Omega)]{\phi^0}^2 +
      \\
      \qquad\;\;
       (\mu^{-1}+\overline{K}\mathcal{C}^2)\norm[H^1(0,\tF;L^2(\Omega)^d)]{\bvec{f}}^2 
      + (2\mu+d\lambda^+)^2 \mu^{-1} \norm[L^2(0,\tF;L^2(\Omega))]{g}^2
      + C_0^{-1}\norm[L^2(0,\tF;L^2(\Omega))]{\lproj[\Omega]{0}g}^2.
    \end{multline}
    Taking $s=\tF$ and denoting by $\mathcal{R}(\bvec{f}, g, \phi^0)$ the quantity in the right-hand side of 
    the previous bound, yields 
    \begin{equation}\label{eq:reg_est.u_C0p}
       2\mu \norm[L^2(0,\tF;H^1(\Omega)^d)]{\dt\bvec{u}}^2+\norm[L^2(0,\tF;L^2(\Omega))]{\dt (C_0 p)}^2
      \lesssim \mathcal{R}(\bvec{f}, g, \phi^0),
    \end{equation}
    Therefore, 
    recalling the definition of the $H^1(0,\tF;V)$-norm and the assumption $(\bvec{u},p)\in L^2(0,\tF;\bvec{U}\times P)$, 
    we obtain $\bvec{u}\in H^1(0,\tF;\bvec{U})$ and $C_0 p \in H^1(0,\tF;L^2(\Omega))$. 
    According to \eqref{eq:compatibility}, in the case $C_0=0$ we have $\lproj[\Omega]{0} p = 0$, whereas, if $C_0>0$, 
    we infer from the Morrey inequality and the boundedness of the mean-value operator that
    \begin{equation}\label{eq:reg_est_avp}
      C_0\norm[{C^0([0,\tF];L^2(\Omega))}]{\lproj[\Omega]{0} p}^2 \lesssim
      \norm[H^1(0,\tF;L^2(\Omega))]{\lproj[\Omega]{0}(C_0 p)}^2 \lesssim
      \norm[H^1(0,\tF;L^2(\Omega))]{C_0 p}^2.
    \end{equation}
    As a result, we obtain $\lproj[\Omega]{0} p\in H^1(0,\tF;P)\subset C^0([0,\tF];P)$. 
    Moreover, from the Poincar\'e--Wirtinger inequality and \eqref{eq:basic_reg_est3} it follows that 
    \begin{equation}\label{eq:reg_est_p}
      \underline{K} \norm[{C^0([0,\tF];H^1(\Omega))}]{p-\lproj[\Omega]{0} p}^2 
      \lesssim \max_{s\in[0,\tF]} \underline{K}\norm[L^2(\Omega)^d]{\GRAD p(s)}^2
      \lesssim \mathcal{R}(\bvec{f}, g, \phi^0).
    \end{equation}
    We conclude by observing that the bounds \eqref{eq:reg_est_avp} and \eqref{eq:reg_est_p} yield $p\in C^0([0,\tF];P)$, 
    while \eqref{eq:reg_est_inf-sup} together with \eqref{eq:reg_est.u_C0p} and \eqref{eq:reg_est_avp} 
    yield $p\in H^1(0,\tF;L^2(\Omega))$.
\end{proof}

\begin{remark}[Regularity of the flux]
  Under the assumptions of Theorem \ref{thm:reg_est}, one can also prove that $-\DIV(\diff\GRAD p)\in L^2(0,\tF ; L^2(\Omega))$. 
  Indeed, from $\bvec{u}\in H^1(0,\tF;\bvec{U})$ and $p\in H^1(0,\tF;L^2(\Omega))$, we infer
  $
  \dt(C_0 p + \DIV\bvec{u})\in L^2(0,\tF; L^2(\Omega))
  $
  and, as a result of \eqref{eq:biot:strong:flow},
  $$
  \DIV(\diff\GRAD p) = g - \dt(C_0 p + \DIV\bvec{u}) \in L^2(0,\tF; L^2(\Omega)).
  $$
  This result will be needed to define the pressure interpolate in \eqref{eq:hat.p}.
\end{remark}

\subsubsection{Initial pore pressure}\label{sec:regularity:p(0)}
The bound \eqref{eq:reg_p_initial} on the gradient of the initial pore pressure is instrumental to proving the regularity results of Theorem \ref{thm:reg_est}. Under some additional assumptions on problem data, we are able to prove that $p(0)$ solving \eqref{eq:initial_mech} and \eqref{eq:weak_form.initial} satisfies \eqref{eq:reg_p_initial}.
We distinguish three cases:
\begin{enumerate}
\item \emph{Assume $\Omega$ convex and $C_0=0$}. Then, $\bvec{u}(0)$ and $p(0)$ solve the well-posed Stokes problem 
  $$
  \begin{alignedat}{2}
    (2\mu\GRADs\bvec{u}(0),\GRADs\bvec{v})_{L^2(\Omega)^{d\times d}}
    + b(\bvec{v},p(0))
    &=
    (\bvec{f}(0)-\lambda\GRAD\phi^0,\bvec{v})_{L^2(\Omega)^d}
    &\qquad&\forall \bvec{v}\in \bvec{U},
    \\
    -b(\bvec{u}(0),r) &=
    (\phi^0,r)_{L^2(\Omega)}
    &\qquad&\forall r\in L^2(\Omega).
  \end{alignedat}
  $$
  It is proved in \cite{Dauge:89} for convex polygonal or polyhedral domains that 
  \[
    \begin{aligned}
      2\mu\norm[H^2(\Omega)^d]{\bvec{u}(0)}
      + \norm[L^2(\Omega)^d]{\GRAD p(0)}
      &\lesssim \norm[L^2(\Omega)^d]{\bvec{f}(0)-\lambda\GRAD\phi^0} + \norm[H^1(\Omega)]{\phi^0}
      \\
      &\lesssim \norm[L^2(\Omega)^d]{\bvec{f}(0)} + (1+|\lambda|)\norm[H^1(\Omega)]{\phi^0}.
    \end{aligned}
  \]
  As a result, estimate \eqref{eq:reg_p_initial} holds.
\item \emph{Assume $\Omega$ convex and $C_0>0$}. We can add the quantity $C_0^{-1}(\DIV\bvec{u}(0),\DIV\bvec{v})_{L^2(\Omega)}$ to both sides of \eqref{eq:initial_mech} and rearrange, thus obtaining
  $$
  \begin{aligned}
    a(\bvec{u}(0),\bvec{v})
    + C_0^{-1}(\DIV\bvec{u}(0),\DIV\bvec{v})_{L^2(\Omega)}
    &= (\bvec{f}(0),\bvec{v})_{L^2(\Omega)^d} + (p(0)+C_0^{-1}\DIV\bvec{u}(0), \DIV\bvec{v})_{L^2(\Omega)}
    \\
    &= (\bvec{f}(0) - C_0^{-1}\GRAD\phi^0 ,\bvec{v})_{L^2(\Omega)},
  \end{aligned}
  $$
  where, to pass to the second line, we have used \eqref{eq:weak_form.initial} and Green's formula. Therefore, we infer that the initial displacement $\bvec{u}(0)$ is the solution of the well-posed linear elasticity problem 
  $$
  (2\mu\GRADs\bvec{u}(0) + (\lambda + C_0^{-1})\DIV\bvec{u}(0)\Id,\GRADs\bvec{v})_{L^2(\Omega)^{d\times d}}
  = (\bvec{f}(0) - C_0^{-1}\GRAD\phi^0 ,\bvec{v})_{L^2(\Omega)^d} \qquad\forall \bvec{v}\in \bvec{U}.
  $$
  Once $\bvec{u}(0)$ is found, the initial pressure $p(0)$ is given by \eqref{eq:biot:strong:initial}.
  It is proved in \cite[Theorem 3.2.1.2]{Grisvard:85} for convex Lipschitz 
  domains that there is $\widetilde{\mathcal{C}}$, depending on $\mu,\lambda$ and $C_0$, such that
  $$
    \norm[H^2(\Omega)^d]{\bvec{u}(0)}
    \lesssim \widetilde{\mathcal{C}}\norm[L^2(\Omega)^d]{\bvec{f}(0)-C_0^{-1}\GRAD\phi^0}
    \le \widetilde{\mathcal{C}} \left( \norm[L^2(\Omega)^d]{\bvec{f}(0)} + C_0^{-1}\norm[H^1(\Omega)]{\phi^0} \right).
  $$
  Hence, it is inferred that $p(0)= C_0^{-1} \left(\phi^0 - \DIV\bvec{u}(0)\right)\in H^1(\Omega)$ 
  and \eqref{eq:reg_p_initial} holds.
  \item \emph{Assume $\diff = \kappa \Id$, with $\kappa:\Omega\to\Real$.} 
  Owing to this assumption, and recalling the boundary condition \eqref{eq:biot:strong:bc.p}, we consider an initial pressure $p(0)$ such that
  \begin{equation}\label{eq:pressure.elliptic_bc}
    \GRAD p(0)\cdot\normal = 0 \qquad\text{on }\partial\Omega.
  \end{equation}
  Since $\bvec{f}\in H^1(0,\tF;L^2(\Omega)^d)$, it is inferred from \eqref{eq:biot:strong:mechanics} that in $H^{-1}(\Omega)$ it holds 
  $$
  \begin{aligned}
    \DIV(\GRAD p(0)) &= \DIV \left(\bvec{f}(0) + \DIV \btens{\sigma}(\GRADs\bvec{u}(0)) \right)
    \\
    &= \DIV \bvec{f}(0) + \DIV \left( (2\mu + \lambda)\GRAD (\DIV \bvec{u}(0)) -\mu\CURL(\CURL\bvec{u}(0)) \right)
    \\
    &= \DIV \bvec{f}(0) + \DIV \left( (2\mu + \lambda)\GRAD (\phi^0 - C_0 p(0)) \right),
  \end{aligned}
  $$
  where we have used $\LAPL\bvec{u}(0) = \GRAD(\DIV\bvec{u}(0)) - \CURL(\CURL\bvec{u}(0))$ to pass to the second line and 
  $\DIV(\CURL\CURL\bvec{u}(0))= \bvec{0}$ together with \eqref{eq:biot:strong:initial} to conclude. 
  Rearranging the previous terms in the equality, we have
  \begin{equation}\label{eq:pressure.elliptic_pb}
    \DIV(\GRAD p(0)) = (1 + 2\mu C_0 + \lambda C_0)^{-1} \DIV \left( \bvec{f}(0) + (2\mu+\lambda)\GRAD\phi^0\right)
    \in H^{-1}(\Omega).
  \end{equation}
  Multiplying \eqref{eq:pressure.elliptic_pb} by $p(0)$, integrating both sides by parts, 
  recalling condition \eqref{eq:pressure.elliptic_bc}, using the Cauchy--Schwarz inequality, simplifying the result, and using the triangle inequality, leads to
  $$
  \begin{aligned}
  \norm[L^2(\Omega)^d]{\GRAD p(0)}
  &\le (1 + 2\mu C_0 + \lambda C_0)^{-1}\norm[L^2(\Omega)^d]{\bvec{f}(0) + (2\mu+\lambda)\GRAD\phi^0}
  \\
  &\le \overline{\mathcal{C}}\left(\norm[L^2(\Omega)^d]{\bvec{f}(0)} + \norm[H^1(\Omega)]{\phi^0}\right),
  \end{aligned}
  $$  
  which is \eqref{eq:reg_p_initial} with $\overline{\mathcal{C}}$ depending on $\mu,\lambda$ and $C_0$.
\end{enumerate}


\section{Abstract framework}\label{sec:abstract.framework}

We consider an abstract analysis framework for the discretisation of the Biot problem \eqref{eq:biot:strong} that enables one to exploit the regularity results of Theorem \ref{thm:reg_est}.
For the sake of simplicity, the focus is on fully couples schemes.

\subsection{Discrete problem}

In the spirit of \cite{Di-Pietro.Droniou:18}, we consider a scheme for the approximation of the Biot problem in fully discrete formulation.

\begin{assumption}[Discrete spaces and forms]\label{ass:discrete.setting}
  We assume the following objects given:
  \begin{enumerate}
  \item \emph{Temporal subdivision.} A regular subdivision of $[0,\tF]$ into $N>0$ time intervals of length $\tau\coloneq\tF/N$.
    For all $0\le n\le N$, we let $t^n\coloneq n\tau$.
    Given a vector space $V$ we define, for all $1\le n\le N$, the backward discrete time derivative operator $\delta_t^n:V^{N+1}\to V$ such that, for any family $\varphi_\tau\coloneq(\varphi^n)_{0\le n\le N}\in V^{N+1}$,
    \[
      \delta_t^n\varphi_\tau\coloneq\frac{\varphi^n-\varphi^{n-1}}{\tau}.
    \]
  \item \emph{Velocity and pore pressure unknowns.} Two finite-dimensional vector spaces $\bvec{\abs{U}}_h$ and $\abs{P}_h$ containing, respectively, the discrete unknowns for the displacement and the pore pressure.
  \item \emph{Pore pressure reconstruction in $L^2(\Omega)$.} A linear operator $\abs{r}_h:\abs{P}_h\to L^2(\Omega)$ enabling the reconstruction of a function in $L^2(\Omega)$ from a set of pore pressure unknowns.
    We will denote by $\abs{L}_h$ the range of $\abs{r}_h$. 
  \item \emph{Bilinear forms.} Three bilinear forms $\abs{a}_h:\bvec{\abs{U}}_h\times\bvec{\abs{U}}_h\to\Real$, $\abs{b}_h:\bvec{\abs{U}}_h\times\abs{P}_h\to\Real$, and $\abs{c}_h:\abs{P}_h\times\abs{P}_h\to\Real$ discretising, respectively, the mechanical term in \eqref{eq:biot:strong:mechanics}, the displacement--pore pressure coupling terms, and the Darcy term in \eqref{eq:biot:strong:flow}.
    The bilinear form $\abs{b}_h$ depends on its second argument only through the reconstruction operator $\abs{r}_h$, that is, there is a bilinear form $\widetilde{\abs{b}}_h:\bvec{\abs{U}}_h\times\abs{L}_h\to\Real$ such that
    \[
    \abs{b}_h(\bvec{v}_h,q_h) = \widetilde{\abs{b}}_h(\bvec{v}_h,\abs{r}_hq_h)\qquad\forall (\bvec{v}_h,q_h)\in\bvec{\abs{U}}_h\times\abs{P}_h.
    \]
    The bilinear form $\abs{c}_h$ is further assumed to be symmetric.
  \item \emph{Source terms and initial condition.} One family of linear forms $(\abs{f}_h^n:\bvec{\abs{U}}_h\to\Real)_{1\le n\le N}$ and one of functions $(\abs{g}^n)_{1\le n\le N}\in L^2(\Omega)^N$, as well as a function $\varphi^0\in L^2(\Omega)$.
    If $C_0=0$, we assume that the following compatibility conditions are verified:
    \begin{equation}\label{eq:abs:compatibility}
      \int_\Omega \varphi^0(\bvec{x})\ud\bvec{x} = 0,\qquad
      \int_\Omega \abs{g}^n(\bvec{x})\ud\bvec{x} = 0\quad\forall 1\le n\le N.
    \end{equation}
  \end{enumerate}
\end{assumption}
A few remarks are of order before proceeding.
\begin{remark}[Symmetry of $\abs{c}_h$]
  The symmetry of the bilinear form $\abs{c}_h$ is used in the proof of Lemma \ref{lem:abs:a-priori}; see, in particular, \eqref{eq:abs.est:2}.
\end{remark}
\begin{remark}[Discrete divergence operator]
  Denote, for any $\bvec{v}_h\in\bvec{\abs{U}}_h$, by $\abs{D}_h\bvec{v}_h\in\abs{L}_h$ the Riesz representation of the linear form $\widetilde{\abs{b}}_h(\bvec{v}_h,\cdot)$ in $\abs{L}_h$ equipped with the usual $L^2$-product, that is,
  \begin{equation}\label{eq:abs.Dh}
    (\abs{D}_h\bvec{v}_h,q_h)_{L^2(\Omega)} = -\widetilde{\abs{b}}_h(\bvec{v}_h,q_h)\qquad\forall q_h\in\abs{L}_h.
  \end{equation}
  The resulting operator $\abs{D}_h:\bvec{\abs{U}}_h\to\abs{L}_h$ can be interpreted as a discrete counterpart of the continuous divergence operator.
\end{remark}
\begin{remark}[Time discretisation]
  Since the emphasis of this work is on space discretisations, we have decided to focus on time discretisations based on a first-order backward difference approximation of the time derivative.
  Higher-order approximations will be considered numerically in Section \ref{sec:numerical.examples}.
  The analysis of more elaborate time stepping strategies is postponed to future works.
\end{remark}
\begin{problem}[Abstract discrete problem]
  The families of discrete displacements $\bvec{u}_{h\tau}\coloneq(\bvec{u}_h^n)_{0\le n\le N}\in\bvec{\abs{U}}_h^{N+1}$ and pore pressures $p_{h\tau}\coloneq(p_h^n)_{0\le n\le N}\in\abs{P}_h^{N+1}$ are such that, for $n=1,\ldots,N$,
  \begin{subequations}\label{eq:biot:discrete}
    \begin{alignat}{2}
      \label{eq:biot:discrete:mechanics}
      \abs{a}_h(\bvec{u}_h^n,\bvec{v}_h) + \abs{b}_h(\bvec{v}_h,p_h^n) &= \abs{f}_h^n(\bvec{v}_h) &\qquad& \forall \bvec{v}_h\in\bvec{\abs{U}}_h,
      \\
      \label{eq:biot:discrete:flow}
      C_0(\abs{r}_h\delta_t^n p_{h\tau},\abs{r}_h q_h)_{L^2(\Omega)} - \abs{b_h}(\delta_t^n \bvec{u}_{h\tau},q_h) + \abs{c}_h(p_h^n,q_h) &= (\abs{g}^n,\abs{r}_h q_h)_{L^2(\Omega)}
      &\qquad& \forall q_h\in\abs{P}_h,
    \end{alignat}
    and the initial values of the discrete displacement and pore pressure satisfy
    \begin{equation}\label{eq:biot:discrete:ic}
      C_0(\abs{r}_hp_h^0,\abs{r}_h q_h)_{L^2(\Omega)} - \abs{b}_h(\bvec{u}_h^0,q_h) = (\varphi^0, \abs{r}_h q_h)_{L^2(\Omega)}\qquad\forall q_h\in\abs{P}_h.
    \end{equation}
    If $C_0=0$, we additionally require
    \begin{equation}\label{eq:biot:discrete:zero.average.p}
      \int_\Omega\abs{r}_hp_h^n=0\qquad\forall 1\le n\le N.
    \end{equation}
  \end{subequations}
\end{problem}

\begin{remark}[Discrete porosity]
  Define the family of discrete porosities $\phi_{h\tau}\coloneq(\phi_h^n)_{0\le n\le N}\in\abs{L}_h^{N+1}$ such that
  \[
    \phi_h^n\coloneq C_0\abs{r}_h p_h^n + \abs{D}_h \bvec{u}_h^n\qquad\forall 0\le n\le N.
  \]
  Recalling \eqref{eq:abs.Dh}, equation \eqref{eq:biot:discrete:flow} can be reformulated as follows:
  \[
  (\delta_t^n\phi_{h\tau},\abs{r}_h q_h)_{L^2(\Omega)} + \abs{c}_h(p_h^n,q_h) = (\abs{g}_h^n,\abs{r}_h q_h)_{L^2(\Omega)}\qquad\forall q_h\in\abs{P}_h,
  \]
  thereby emphasising the fact that the discrete porosity is the quantity under the discrete time derivative.
  Hence, to enforce an initial condition, we do not need to prescribe separately an initial discrete displacement $\bvec{u}_h^0\in\bvec{\abs{U}}_h$ and pore pressure $p_h^0\in\abs{P}_h$, but rather the combination of these quantities corresponding to $\phi_h^0$.
  The requirement on the discrete initial condition \eqref{eq:biot:discrete:ic} translates into
  \[
  (\phi_h^0, q_h)_{L^2(\Omega)} = (\varphi^0, q_h)_{L^2(\Omega)}\qquad\forall q_h\in\abs{L}_h,
  \]
  meaning that $\phi_h^0$ is the $L^2$-orthogonal projection of the continuous initial porosity $\varphi^0$ on $\abs{L}_h$.
\end{remark}

With the developments of the next section in mind, we reformulate the abstract discrete problem by integrating \eqref{eq:biot:discrete:flow} in time.
At the discrete level, this is done by summing, for all $1\le i\le n$, \eqref{eq:biot:discrete:flow} multiplied by the timestep $\tau$. 
Introducing, for all $1\le n\le N$, the families $s_{h\tau}\coloneq(s_h^n)_{0\le n\le N}\in\abs{P}_h^{N+1}$ and $(\abs{G}^n)_{1\le n\le N}\in L^2(\Omega)^N$ such that
$$
\text{%
  $s_h^0 \coloneq 0$, 
  $s_h^n \coloneq \sum_{i=1}^n \tau p_h^i$ and
  $\abs{G}^n\coloneq \sum_{i=1}^n \tau \abs{g}^i$ for all $1\le n\le N$,
}
$$
it is inferred that, for all $n=1,\ldots,N$, the solution of \eqref{eq:biot:discrete} satisfies, for all $(\uvec{v}_h,q_h)\in\bvec{\abs{U}}_h\times\abs{P}_h$,
\[
\begin{aligned}
  \abs{a}_h(\bvec{u}_h^n,\bvec{v}_h) + \abs{b}_h(\bvec{v}_h,p_h^n) & = \abs{f}_h^n(\bvec{v}_h),
  \\
  C_0(\abs{r}_h p_h^n,\abs{r}_h q_h)_{L^2(\Omega)}-\abs{b_h}(\bvec{u}^n_h,q_h)+\tau\abs{c}_h(p_h^n,q_h) &
  = (\abs{G}^n+\varphi^0,\abs{r}_h q_h)_{L^2(\Omega)}
  - \abs{c}_h(s_h^{n-1},q_h).
\end{aligned}
\]
Summing the two equations above we infer that, for $n=1,\ldots,N$, $(\bvec{u}_h^n,p_h^n)\in\bvec{\abs{U}}_h\times\abs{P}_h$ solves
\begin{equation}\label{eq:biot:discrete2}
  \abs{A}_h \left((\bvec{u}_h^n,p_h^n),(\bvec{v}_h,q_h) \right) = \abs{F}_h^n(\bvec{v}_h,q_h)
  \qquad \forall (\bvec{v}_h,q_h)\in\bvec{\abs{U}}_h\times\abs{P}_h, 
\end{equation}
where we have defined the total bilinear form $\abs{A}_h:(\bvec{\abs{U}}_h\times\abs{P}_h)\times (\bvec{\abs{U}}_h\times\abs{P}_h)\to\Real$ such that, for all $(\bvec{u}_h,p_h),(\bvec{v}_h,q_h)\in\bvec{\abs{U}}_h\times\abs{P}_h$, 
\begin{equation*}
  \abs{A}_h \big((\bvec{u}_h,p_h), (\bvec{v}_h,q_h) \big)
  \coloneq  
  \abs{a}_h(\bvec{u}_h,\bvec{v}_h) + \abs{b}_h(\bvec{v}_h,p_h) - \abs{b}_h(\bvec{u}_h,q_h)
  + C_0(\abs{r}_h p_h,\abs{r}_h q_h)_{L^2(\Omega)} +\tau\abs{c}_h(p_h,q_h),
\end{equation*}
and the linear form
\begin{equation*}
  \abs{F}_h^n(\bvec{v}_h,q_h) \coloneq
  \abs{f}_h^n(\bvec{v}_h) + (\abs{G}^n+\varphi^0,\abs{r}_h q_h)_{L^2(\Omega)} - \abs{c}_h(s_h^{n-1},q_h). 
\end{equation*}
Notice that $\abs{F}_h^n$ depends only on the problem data and on the (available) values of the discrete pressure at 
times $t^i$, $0\le i\le n-1$, and it therefore appears legitimately in the right-hand side of the equation defining 
$(\bvec{u}_h^n,p_h^n)$.

\subsection{A priori estimate}

We next introduce a set of assumptions that enable us to derive an a priori estimate for problem \eqref{eq:biot:discrete}.
Such an estimate has two purposes:
on the one hand, when the families of forms $(\abs{f}_h^n)_{1\le n\le N}$ and functions $(\abs{g}_h^n)_{1\le n\le N}$ in Assumption \ref{ass:discrete.setting} are discrete counterparts of the source terms $\bvec{f}$ and $g$ in \eqref{eq:biot:strong:mechanics} and \eqref{eq:biot:strong:flow}, respectively, and $\varphi^0=\phi^0$, it establishes the well-posedness of the discrete problem \eqref{eq:biot:discrete};
on the other hand, when these quantities represent, respectively, the residuals of the mechanics and flow equations at each discrete time and the residual of the initial condition, it gives a basic error estimate.
\begin{assumption}[Stability]\label{ass:stability}
  With the same notations as in Assumption \ref{ass:discrete.setting}, we assume the following:
  \begin{enumerate}
  \item \emph{Discrete norms.} We furnish $\bvec{\abs{U}}_h$ with the norm $\norm[\bvec{\abs{U}},h]{{\cdot}}$ and $\abs{P}_h$ with the seminorm $\norm[\abs{P},h]{{\cdot}}$, which is assumed to be a norm on $\abs{P}_{h,0}\coloneq\left\{q_h\in\abs{P}_h\st\int_\Omega\abs{r}_hq_h=0\right\}$.
    The norm dual to $\norm[\bvec{\abs{U}},h]{{\cdot}}$ is the mapping $\norm[\bvec{\abs{U}},h,*]{{\cdot}}$ such that, for any linear form $\ell_h:\bvec{\abs{U}}_h\to\Real$,
    \begin{equation}\label{eq:abs:norm.Uh*}
      \norm[\bvec{\abs{U}},h,*]{\ell_h}\coloneq\sup_{\bvec{v}_h\in\bvec{\abs{U}}_h\setminus\{\bvec{0}\}}\frac{\ell_h(\bvec{v}_h)}{\norm[\bvec{\abs{U}},h]{\bvec{v}_h}}.
    \end{equation}
  \item \emph{Coercivity and boundedness of $\abs{a}_h$ and $\abs{c}_h$.}
    There exist two real numbers $0<\underline{\alpha}\le\overline{\alpha}$ such that
    \begin{alignat}{2}\label{eq:abs.ah:coercivity}
      \underline{\alpha}\norm[\bvec{\abs{U}},h]{\bvec{v}_h}^2&\le\abs{a}_h(\bvec{v}_h,\bvec{v}_h) &\qquad&\forall \bvec{v}_h\in\bvec{\abs{U}}_h,
      \\ \label{eq:abs.ah:boundedness}
      \abs{a}_h(\bvec{w}_h,\bvec{v}_h)&\le\overline{\alpha}\norm[\bvec{\abs{U}},h]{\bvec{w}_h}\norm[\bvec{\abs{U}},h]{\bvec{v}_h} &\qquad&\forall \bvec{w}_h,\bvec{v}_h\in\bvec{\abs{U}}_h,
    \end{alignat}
    and there exists $\gamma>0$ such that
    \begin{equation}\label{eq:abs.ch:coercivity}
      \gamma\norm[\abs{P},h]{q_h}^2\le\abs{c}_h(q_h,q_h)\qquad\forall q_h\in\abs{P}_h,
    \end{equation}
  \item \emph{Inf-sup stability of $\abs{b}_h$.}
    There exists $\beta>0$ such that
    \begin{equation}\label{eq:abs.bh:inf-sup}
      \text{
        $\beta\norm[L^2(\Omega)]{\abs{r}_hq_h-\lproj[\Omega]{0}\abs{r}_hq_h}
        \le\norm[\bvec{\abs{U}},h,*]{\abs{b}_h(\cdot,q_h)}$
        for all $q_h\in\abs{P}_h$.
      }
    \end{equation}
  \end{enumerate}
\end{assumption}
\begin{remark}[Uniform stability]
  In practice, deriving optimal error estimates requires that the real numbers $\underline{\alpha}$, $\overline{\alpha}$, $\beta$, and $\gamma$ are independent of discretisation parameters such as the (temporal and) spatial meshsize.
\end{remark}
\begin{lemma}[Abstract a priori estimate]\label{lem:abs:a-priori}
  Let Assumptions \ref{ass:discrete.setting} and \ref{ass:stability} hold true, and let $(\bvec{u}_{h\tau},p_{h\tau})\in\bvec{\abs{U}}_h^{N+1}\times\abs{P}_h^{N+1}$ solve \eqref{eq:biot:discrete}.
  Then, it holds
  \begin{multline}\label{eq:abs:a-priori}
    \underline{\alpha}\sum_{n=1}^N\tau\norm[\bvec{\abs{U}},h]{\bvec{u}_h^n}^2
    + \sum_{n=1}^N\tau\left(
    C_0\norm[L^2(\Omega)]{\abs{r}_hp_h^n}^2
    + \frac{\beta^2\underline{\alpha}}{2\overline{\alpha}^2}\norm[L^2(\Omega)]{\abs{r}_hp_h^n-\lproj[\Omega]{0}\abs{r}_hp_h^n}^2
    \right)
    + \gamma\norm[\abs{P},h]{s_h^N}^2
    \le
    \mathcal{C}_1\sum_{n=1}^N\tau\norm[\bvec{\abs{U}},h,*]{\abs{f}_h^n}^2
    \\
    + \mathcal{C}_2\left(
    \sum_{n=1}^N\tau\norm[L^2(\Omega)]{\abs{G}^n}^2
    + \tF\norm[L^2(\Omega)]{\varphi^0}^2
    \right)
    + \mathcal{C}_3\left(
    \sum_{n=1}^N\tau\norm[L^2(\Omega)]{\lproj[\Omega]{0}\abs{G}^n}^2
    + \tF\norm[L^2(\Omega)]{\lproj[\Omega]{0}\varphi^0}^2
    \right),
  \end{multline}
  where
  \begin{equation}\label{eq:C1.C2.C3}
    \mathcal{C}_1\coloneq\frac{4\overline{\alpha}^2+2\underline{\alpha}^2}{\underline{\alpha}\overline{\alpha}^2},\qquad
    \mathcal{C}_2\coloneq\frac{16\overline{\alpha}^2}{\underline{\alpha}\beta^2},\qquad
    \mathcal{C}_3\coloneq\frac4{C_0}.
  \end{equation}
\end{lemma}
\begin{remark}[A priori estimate for $C_0=0$]
  When $C_0=0$, the a priori bound \eqref{eq:abs:a-priori} remains valid with the convention that $C_0^{-1}\sum_{n=1}^N\tau\norm[L^2(\Omega)]{\lproj[\Omega]{0}\abs{G}^n}^2=0$ and $C_0^{-1}\norm[L^2(\Omega)]{\lproj[\Omega]{0}\varphi^0}^2=0$.
  These conventions are justified by the compatibility conditions \eqref{eq:abs:compatibility}.
\end{remark}
\begin{proof}[Proof of Lemma \ref{lem:abs:a-priori}]
  We start by obtaining an estimate for the pore pressure reconstruction.
  Using the inf-sup condition \eqref{eq:abs.bh:inf-sup}, recalling the definition \eqref{eq:abs:norm.Uh*} of the dual norm, and expressing $\abs{b}_h(\uvec{v}_h,p_h^n)$ according to \eqref{eq:biot:discrete:mechanics}, we have, for all $1\le n\le N$,
  \[
  \beta\norm[L^2(\Omega)]{\abs{r}_hp_h^n-\lproj[\Omega]{0}\abs{r_h}p_h^n}
  \le\!\sup_{\bvec{v}_h\in\bvec{\abs{U}}_h\setminus\{\bvec{0}\}}\!\frac{\abs{b}_h(\bvec{v}_h,p_h^n)}{\norm[\bvec{\abs{U}},h]{\bvec{v}_h}}
  =\!\sup_{\bvec{v}_h\in\bvec{\abs{U}}_h\setminus\{\bvec{0}\}}\!\frac{\abs{f}_h^n(\bvec{v}_h) - \abs{a}_h(\bvec{u}_h^n,\bvec{v}_h)}{\norm[\bvec{\abs{U}},h]{\bvec{v}_h}}
  \le\norm[\bvec{\abs{U}},h,*]{\abs{f}_h^n} + \overline{\alpha}\norm[\bvec{\abs{U}},h]{\bvec{u}_h^n},
  \]
  where we have used again the definition \eqref{eq:abs:norm.Uh*} of the dual norm and the boundedness \eqref{eq:abs.ah:boundedness} of the bilinear form $\abs{a}_h$ to conclude.
  Squaring the above expression and recalling that $(a+b)^2\le 2(a^2+b^2)$ for any $a,b\in\Real$, we get
  \begin{equation}\label{eq:abs.est:pressure}
    \beta^2\norm[L^2(\Omega)]{\abs{r}_hp_h^n-\lproj[\Omega]{0}\abs{r_h}p_h^n}^2
    \le2\left(
    \norm[\bvec{\abs{U}},h,*]{\abs{f}_h^n}^2 + \overline{\alpha}^2\norm[\bvec{\abs{U}},h]{\bvec{u}_h^n}^2
    \right)\qquad\forall 1\le n\le N.
  \end{equation}

  Next, for any $1\le n\le N$, taking $(\bvec{v}_h, q_h)=(\bvec{u}_h^n, p_h^n)$ in \eqref{eq:biot:discrete2}, expanding $s_h^{n-1}$ according to its definition, rearranging, and using the coercivity \eqref{eq:abs.ah:coercivity} of $\abs{a}_h$, we obtain
  \begin{equation}\label{eq:abs.est:1}
     \underline{\alpha}\norm[\bvec{\abs{U}},h]{\bvec{u}_h^n}^2
     + C_0\norm[L^2(\Omega)]{\abs{r}_hp_h^n}^2
     + \sum_{i=1}^n\tau\abs{c}_h(p_h^i,p_h^n)
     =
     \abs{f}_h^n(\bvec{u}_h^n) 
     +(\abs{G}^n + \varphi^0,\abs{r}_hp_h^n)_{L^2(\Omega)}.
  \end{equation}
  Using the linearity of $\abs{c}_h$, the definition of $s_h^n$, and observing that $p_h^n=\delta_t^ns_{h\tau}$, we infer
  \begin{equation}\label{eq:abs.est:2}
    \sum_{i=1}^n\tau\abs{c}_h(p_h^i,p_h^n)
    =\abs{c}_h(s_h^n,\delta_t^n s_{h\tau})
    =\frac1{2\tau}\left(
    \abs{c}_h(s_h^n,s_h^n) + \tau^2\abs{c}_h(\delta_t^n s_{h\tau},\delta_t^n s_{h\tau}) - \abs{c}_h(s_h^{n-1},s_h^{n-1})
    \right),
  \end{equation}
  where the conclusion follows from the symmetry of $\abs{c}_h$.
  Plugging \eqref{eq:abs.est:2} into \eqref{eq:abs.est:1}, and further observing that $\abs{c}_h(\delta_t^n s_{h\tau},\delta_t^n s_{h\tau})\ge 0$ by the coercivity \eqref{eq:abs.ch:coercivity} of $\abs{c}_h$, we get, for any $1\le n\le N$,
  \[
    \underline{\alpha}\norm[\bvec{\abs{U}},h]{\bvec{u}_h^n}^2
    + C_0\norm[L^2(\Omega)]{\abs{r}_hp_h^n}^2
    + \frac1{2\tau}\left(\abs{c}_h(s_h^n,s_h^n) - \abs{c}_h(s_h^{n-1},s_h^{n-1}) \right)
    \le
    \abs{f}_h^n(\bvec{u}_h^n)
    + (\abs{G}^n + \varphi^0,\abs{r}_hp_h^n)_{L^2(\Omega)}.
  \]
  We  now multiply by $\tau$, sum over $1\le n\le N$, telescope out the appropriate terms, recall that $s_h^0=0$ by definition, and use the coercivity \eqref{eq:abs.ch:coercivity} of $\abs{c}_h$ to infer
  \begin{equation}\label{eq:abs.err:basic}
    \underline{\alpha}\sum_{n=1}^N\tau\norm[\bvec{\abs{U}},h]{\bvec{u}_h^n}^2
    + C_0\sum_{n=1}^N\tau\norm[L^2(\Omega)]{\abs{r}_hp_h^n}^2
    + \frac{\gamma}{2}\norm[\abs{P},h]{s_h^N}^2
    \le
    \sum_{n=1}^N\tau\abs{f}_h^n(\bvec{u}_h^n)
    + \sum_{n=1}^N\tau(\abs{G}^n + \varphi^0,\abs{r}_hp_h^n)_{L^2(\Omega)}.
  \end{equation}
  
  We next proceed to bound the terms in the right-hand side, denoted for short $\term_1$ and $\term_2$.
  For the first term, using the definition \eqref{eq:abs:norm.Uh*} of dual norm  followed by a generalized Young inequality, we can write, for any $\lambda_1>0$,
  \begin{equation}\label{eq:abs.est:T1}
    \term_1\le\frac{\lambda_1}{2}\sum_{n=1}^N\tau\norm[\bvec{\abs{U}},h]{\bvec{u}_h^n}^2
    + \frac{1}{2\lambda_1}\sum_{n=1}^N\tau\norm[\bvec{\abs{U}},h,*]{\abs{f}_h^n}^2.
  \end{equation}
  The second term requires a careful treatment in order to achieve robustness for $C_0=0$. The starting point is the following splitting:
  \[
  \term_2 =
  \sum_{n=1}^N\tau(\abs{G}^n+\varphi^0,\abs{r}_hp_h^n-\lproj[\Omega]{0}\abs{r}_hp_h^n)_{L^2(\Omega)}
  + \sum_{n=1}^N\tau(\abs{G}^n+\varphi^0,\lproj[\Omega]{0}\abs{r}_hp_h^n)_{L^2(\Omega)}
  \eqcolon \term_{2,1} + \term_{2,2}.
  \]
  For the first contribution, using the Cauchy--Schwarz, generalized Young, and triangle inequalities and continuing with \eqref{eq:abs.est:pressure}, we have, for any $\lambda_2>0$,
  \begin{equation}\label{eq:abs.est:T21}
    \term_{2,1}
    \le\frac{1}{\lambda_2}\left(
    \sum_{n=1}^N\tau\norm[L^2(\Omega)]{\abs{G}^n}^2 + \tF\norm[L^2(\Omega)]{\varphi^0}^2
    \right)
    + \lambda_2 \beta^{-2}\sum_{n=1}^N\tau\left(
    \norm[\bvec{\abs{U}},h,*]{\abs{f}_h^n}^2
    + \overline{\alpha}^2\norm[\bvec{\abs{U}},h]{\bvec{u}_h^n}^2
    \right).
  \end{equation}
  For the second contribution, noticing that, for any $\xi,\eta\in L^2(\Omega)$, $(\lproj[\Omega]{0}\xi,\eta)_{L^2(\Omega)}=(\xi,\lproj[\Omega]{0}\eta)_{L^2(\Omega)}$, we can write
  \begin{equation}\label{eq:abs.est:T22}
    \term_{2,2}
    = \sum_{n=1}^N\tau(\lproj[\Omega]{0}\abs{G}^n+\lproj[\Omega]{0}\varphi^0,\abs{r}_hp_h^n)_{L^2(\Omega)}
    \le\frac{1}{\lambda_3}\left(
    \sum_{n=1}^N\tau\norm[L^2(\Omega)]{\lproj[\Omega]{0}\abs{G}^n}^2
    \hspace{-0.5ex}+ \tF\norm[L^2(\Omega)]{\lproj[\Omega]{0}\varphi^0}^2
    \right) + \frac{\lambda_3}{2}\sum_{n=1}^N\tau\norm[L^2(\Omega)]{\abs{r}_hp_h^n}^2,
  \end{equation}
  where we have used again the Cauchy--Schwarz, generalized Young, and triangle inequalities, the second with weight $\lambda_3>0$.

  Plug, into \eqref{eq:abs.err:basic}, \eqref{eq:abs.est:T1} with $\lambda_1=\underline{\alpha}/2$, \eqref{eq:abs.est:T21} with $\lambda_2=(\beta^2\underline{\alpha})/(4\overline{\alpha}^2)$, and \eqref{eq:abs.est:T22} with $\lambda_3=C_0$ to obtain, after multiplying by 2 the resulting inequality,
  \begin{multline}\label{eq:abs:a-priori:simplified}
    \underline{\alpha}\sum_{n=1}^N\tau\norm[\bvec{\abs{U}},h]{\bvec{u}_h^n}^2
    + C_0\sum_{n=1}^N\tau\norm[L^2(\Omega)]{\abs{r}_hp_h^n}^2
    + \gamma\norm[\abs{P},h]{s_h^N}^2
    \le
    \frac{4\overline{\alpha}^2+\underline{\alpha}^2}{2\underline{\alpha}\overline{\alpha}^2}\sum_{n=1}^N\tau\norm[\bvec{\abs{U}},h,*]{\abs{f}_h^n}^2
    \\
    + \frac{8\overline{\alpha}^2}{\underline{\alpha}\beta^2}\left(
    \sum_{n=1}^N\tau\norm[L^2(\Omega)]{\abs{G}^n}^2
    + \tF\norm[L^2(\Omega)]{\varphi^0}^2
    \right)
    + \frac{2}{C_0}\left(
    \sum_{n=1}^N\tau\norm[L^2(\Omega)]{\lproj[\Omega]{0}\abs{G}^n}^2
    + \tF\norm[L^2(\Omega)]{\lproj[\Omega]{0}\varphi^0}^2
    \right).
  \end{multline}
  Multiplying \eqref{eq:abs.est:pressure} by $\tau\underline{\alpha}/(2\overline{\alpha}^2)$, summing over $1\le n\le N$, using \eqref{eq:abs:a-priori:simplified} to further bound the second term in the right-hand side, and adding the resulting inequality to \eqref{eq:abs:a-priori:simplified} concludes the proof.
\end{proof}

\subsection{Error estimates}

We next derive a basic error estimate, obtained comparing the solution of the discrete problem \eqref{eq:biot:discrete} with an interpolate of the exact solution, whose features are summarised in the following assumption.

\begin{assumption}[Interpolate of the exact solution]\label{ass:interpolate}
  With the same notations as in Assumption \ref{ass:discrete.setting}, we assume given the following objects:
  \begin{enumerate}
  \item \emph{Displacement interpolate.} $\hat{\bvec{u}}_{h\tau}\coloneq(\hat{\bvec{u}}_h^n)_{0\le n\le N}\in\bvec{\abs{U}}_h^{N+1}$ is a family of elements of the discrete displacement space such that, for all $0\le n\le N$, $\hat{\bvec{u}}_h^n$ encodes meaningful information on the exact displacement field at time $t^n$.
  \item \emph{Pore pressure interpolate.}
    $\mathfrak{p}_h\in L^2(0,\tF;\abs{P}_h)$ solves, for a.e. $t\in(0,\tF)$, the problem
    \begin{equation}\label{eq:hat.p}
      \begin{aligned}
        \abs{c}_h(\mathfrak{p}_h(t),q_h) &= 
        -(\DIV(\diff\GRAD p(\cdot,t)),\abs{r}_hq_h)_{L^2(\Omega)}\qquad\forall q_h\in\abs{P}_h
        \\
        \int_\Omega(\abs{r}_h\mathfrak{p}_h)(\bvec{x},t)\ud\bvec{x} &=\int_\Omega p(\bvec{x},t)\ud\bvec{x}.
      \end{aligned}
    \end{equation}
    For all $1\le n\le N$, we assume given a linear time interpolator $\abs{I}^n:L^2(t^{n-1},t^n;V)\to V$, with $V$ a generic vector 
    space, and we set 
    \[
    \hat{p}_h^n\coloneq \abs{I}^n\mathfrak{p}_h.
    \]
    We also take $\hat{p}_h^0\in \abs{P}_h$ such that 
    $(\abs{r}_h\hat{p}_h^0,q_h)_{L^2(\Omega)}= (p(0),q_h)_{L^2(\Omega)}$ for all $q_h\in\abs{L}_h$, and define 
    $$
    \hat{p}_{h\tau}\coloneq(\hat{p}_h^n)_{0\le n\le N}\in\abs{P}_h^{N+1}.
    $$
  \end{enumerate}
\end{assumption}
The errors $\bvec{e}_{h\tau}\coloneq(\bvec{e}_h^n)_{0\le n\le N}\in\bvec{\abs{U}}_h^{N+1}$ on the displacement and $\epsilon_{h\tau}\coloneq(\epsilon_h^n)_{0\le n\le N}\in\abs{P}_h^{N+1}$ on the pore pressure are such that, for all $0\le n\le N$,
\begin{equation}\label{eq:errors}
  \bvec{e}_h^n\coloneq\bvec{u}_h^n-\hat{\bvec{u}}_h^n,\qquad
  \epsilon_h\coloneq p_h^n-\hat{p}_h^n.
\end{equation}

\begin{lemma}[Abstract error estimate]\label{lem:abs:err.est}
  Under Assumptions \ref{ass:discrete.setting}, \ref{ass:stability}, and \ref{ass:interpolate}, we have the following estimate for the errors defined by \eqref{eq:errors}:
  \begin{multline}\label{eq:abs:err.est}
    \underline{\alpha}\sum_{n=1}^N\tau\norm[\bvec{\abs{U}},h]{\bvec{e}_h^n}^2
    + \sum_{n=1}^N\tau\left(
    C_0\norm[L^2(\Omega)]{\abs{r}_h\epsilon_h^n}^2
    + \frac{\beta^2\underline{\alpha}}{2\overline{\alpha}^2}\norm[L^2(\Omega)]{\abs{r}_h\epsilon_h^n-\lproj[\Omega]{0}\abs{r}_h\epsilon_h^n}^2
    \right)
    + \gamma\norm[\abs{P},h]{z_h^N}^2
    \\
    \le
    \mathcal{C}_1\sum_{n=1}^N\tau\norm[\bvec{\abs{U}},h,*]{\mathcal{E}_{\bvec{\abs{U}},h}^n}^2
    + \mathcal{C}_2\left(
    \sum_{n=1}^N\tau\norm[L^2(\Omega)]{\mathcal{E}_{\abs{P},h\tau}^n}^2
    + \tF\norm[L^2(\Omega)]{\mathcal{E}_{0,h}}^2
    \right)
    \\
    + \mathcal{C}_3\left(
    \sum_{n=1}^N\tau\norm[L^2(\Omega)]{\lproj[\Omega]{0}\mathcal{E}_{\abs{P},h\tau}^n}^2
    + \tF\norm[L^2(\Omega)]{\lproj[\Omega]{0}\mathcal{E}_{0,h}}^2
    \right),
  \end{multline}
  where $\mathcal{C}_1$, $\mathcal{C}_2$, and $\mathcal{C}_3$ are defined by \eqref{eq:C1.C2.C3},
  $z_h^0\coloneq 0$, $z_h^n\coloneq\sum_{i=1}^n\tau\epsilon_h^i$ for all $1\le n\le N$,
  and the residuals $\mathcal{E}_{0,h}^n\in\abs{L}_h$ on the initial condition,
  $(\mathcal{E}_{\bvec{\abs{U}},h}^n:\bvec{\abs{U}}_h\to\Real)_{1\le n\le N}$ on the mechanical equilibrium,
  and $(\mathcal{E}_{\abs{P},h\tau}^n)_{1\le n\le N}\in\abs{L}_h^N$ on the flow equation are such that
  $$
    (\mathcal{E}_{0,h},\abs{r}_hq_h)_{L^2(\Omega)}\coloneq
    (\varphi^0,\abs{r}_hq_h)_{L^2(\Omega)}
    - C_0(\abs{r}_h\hat{p}_h^0,\abs{r}_hq_h)_{L^2(\Omega)}
    + \abs{b}_h(\hat{\bvec{u}}_h^0,q_h)
    \qquad\forall q_h\in\abs{P}_h
  $$
  and, for all $1\le n\le N$,
  \begin{alignat}{2}
    \label{eq:abs:EUhn}
    \mathcal{E}_{\bvec{\abs{U}},h}^n(\bvec{v}_h) &\coloneq
    \abs{f}_h^n(\bvec{v}_h)
    - \abs{a}_h(\hat{\bvec{u}}_h^n,\bvec{v}_h)
    - \abs{b}_h(\bvec{v}_h,\hat{p}_h^n)
    &\qquad&\forall\bvec{v}_h\in\bvec{\abs{U}}_h,
    \\ \nonumber
    (\mathcal{E}_{\abs{P},h\tau}^n,\abs{r}_hq_h)_{L^2(\Omega)} &\coloneq
    \sum_{i=1}^n\tau\Big[
      (\abs{g}^i + \DIV(\diff\GRAD\abs{I}^ip),\abs{r}_hq_h)_{L^2(\Omega)}
      \\ \label{eq:abs:EPhtaun}    
      &\qquad\qquad - C_0(\abs{r}_h\delta_t^i\hat{p}_{h\tau},\abs{r}_h q_h)_{L^2(\Omega)}
      + \abs{b}_h(\delta_t^i\hat{\bvec{u}}_{h\tau},q_h)
      \Big]
    &\qquad&\forall q_h\in\abs{P}_h.
  \end{alignat}
\end{lemma}
\begin{proof}
  The errors $(\bvec{e}_{h\tau},\epsilon_{h\tau})\in\bvec{\abs{U}}_h^{N+1}\times\abs{P}_h^{N+1}$ solve problem \eqref{eq:biot:discrete} with $\varphi^0$ replaced by $\mathcal{E}_{0,h}$ and, for all $1\le n\le N$, $\abs{f}_h^n$ replaced by $\mathcal{E}_{\bvec{\abs{U}},h}^n$ and $\abs{g}^n$ replaced by $\varepsilon_{\abs{P},h}^n\in\abs{L}_h$ such that
  \[
  (\varepsilon_{\abs{P},h}^n,\abs{r}_hq_h)_{L^2(\Omega)}
  = (\abs{g}^n + \DIV(\diff\GRAD\abs{I}^n p),\abs{r}_hq_h)_{L^2(\Omega)}
  - C_0(\abs{r}_h\delta_t^n\hat{p}_{h\tau},\abs{r}_hq_h)_{L^2(\Omega)}
  + \abs{b}_h(\delta_t^n\hat{\bvec{u}}_{h\tau},q_h)
  \quad\forall q_h\in\abs{P}_h.
  \]
  The conclusion follows applying the abstract a priori estimate \eqref{eq:abs:a-priori} to this error equation and observing that, with the above choice, $\mathcal{E}_{\abs{P},h\tau}^n$ plays the role of $\abs{G}^n$.
\end{proof}
\begin{remark}[Displacement interpolate and discretisation of the loading and source terms]
  Having introduced, for any $1\le n\le N$, the time interpolator $\abs{I}^n$, a natural (but not mandatory) choice for the displacement interpolate, the discrete loading term, and the discrete source term consists in setting $\hat{\bvec{u}}_h^n\coloneq\abs{I}^n(\bvec{\abs{I}}_{\bvec{\abs{U}},h}\bvec{u})$, $\bvec{\abs{f}}^n=\abs{I}^n\bvec{f}$, and $\abs{g}^n=\abs{I}^n g$, with $\bvec{\abs{I}}_{\bvec{\abs{U}},h}$ denoting a spatial interpolator on $\bvec{\abs{U}}_h$.
  With this choice, the expression \eqref{eq:abs:EPhtaun} of the residual on the flow equation can be simplified observing that $\abs{g}^i+\DIV(\diff\GRAD\abs{I}^ip) = C_0\abs{I}^i(d_t p) + \abs{I}^i\left(d_t(\DIV\bvec{u})\right)$.
\end{remark}


\section{Discrete setting}\label{sec:discrete.setting}

In this section, we describe the discrete setting for HHO space discretisations of problem \eqref{eq:biot:strong}.

\subsection{Space mesh}

We consider spatial meshes corresponding to couples $\Mh\coloneq(\Th,\Fh)$, where $\Th$ is a finite collection of polyhedral elements such that $h\coloneq\max_{T\in\Th}h_T>0$ with $h_T$ denoting the diameter of $T$, while $\Fh$ is a finite collection of planar faces. It is assumed henceforth that the mesh $\Mh$ matches the geometrical requirements detailed in \cite[Chapters 1 and 7]{Di-Pietro.Droniou:20}.
This covers, essentially, any reasonable partition of $\Omega$ into polyhedral sets, not necessarily convex.
For every mesh element $T\in\Th$, we denote by $\Fh[T]$ the subset of $\Fh$ containing the faces that lie on the boundary $\partial T$ of $T$.
For any mesh element $T\in\Th$ and each face $F\in\Fh[T]$, $\normal_{TF}$ is the unit vector normal to $F$ pointing out of $T$.
Boundary faces lying on $\partial\Omega$ and internal faces contained in
$\Omega$ are collected in the sets $\Fhb$ and $\Fhi$, respectively.
For any $F\in\Fhi$, we denote by $T_1$ and $T_2$ the elements of $\Th$ such that $F\subset\partial T_1\cap\partial T_2$.
The numbering of $T_1$ and $T_2$ is assumed arbitrary but fixed, and we set $\normal_F\coloneq\normal_{T_1F}$.
Since we focus on $h$-convergence, it is assumed that the mesh belongs to a regular sequence in the sense of \cite[Definition 1.9]{Di-Pietro.Droniou:20}.

\subsection{Assumptions on the data and mesh compliance}

For the sake of simplicity, throughout the rest of the paper we assume homogeneous and isotropic mechanical properties, so that \eqref{eq:C.lame} holds.
We additionally assume that the permeability coefficient $\diff$ is piecewise constant on a fixed partition $P_\Omega=\{\Omega_i\}_{1\le i\le N_\Omega}$ of the domain into polyhedra.
It is assumed that the mesh $\Mh$ is compliant with the partition $P_\Omega$, that is, for every mesh element $T\in\Th$ there exists a unique index $1\le i\le N_\Omega$ such that $T\subset\Omega_i$.
This implies, in particular, that
\[
\diff_T\coloneq\diff_{|T}\in\Real_\symm^{d\times d}\qquad\forall T\in\Th.
\]
We respectively denote by $\overline{K}_T$ and $\underline{K}_T$ the largest and smallest eigenvalue of $\diff_T$, and define, for all $T\in\Th$, the local anisotropy ratio
\[
\rho_T\coloneq\frac{\overline{K}_T}{\underline{K}_T}.
\]
The global anisotropy ratio is defined as follows (notice that this value is independent of the mesh):
\[
\rho\coloneq\max_{T\in\Th}\rho_T.
\]

\subsection{Local and broken spaces and projectors}

Let a polynomial degree $l\ge 0$ be fixed.
For all $X\in\Th\cup\Fh$, denote by $\Poly{l}(X)$ the space spanned by the restriction to $X$ of $d$-variate polynomials of total degree $\le l$, and let $\lproj[X]{l}:L^1(X)\to\Poly{l}(X)$ be the corresponding $L^2$-orthogonal projector such that, for any $v\in L^1(X)$,
\[
(\lproj[X]{l}v-v,w)_{L^2(X)}=0\qquad\forall w\in\Poly{l}(X).
\]
The vector version $\vlproj[X]{l}:L^1(X)^d\to\Poly{l}(X)^d$, is obtained applying $\lproj[X]{l}$ component-wise.
For any $T\in\Th$ we will also need, in what follows, the space $\Polys{l}(T)$ of $d\times d$ symmetric matrix-valued fields on $T$ with polynomial components.

At the global level, we introduce the broken polynomial space
\[
\Poly{l}(\Th)\coloneq\left\{v\in L^1(\Omega)\st v_{|T}\in\Poly{l}(T)\quad\forall T\in\Th\right\}
\]
along with the corresponding vector version $\Poly{l}(\Th)^d$ and the space $\Polys{l}(\Th)$ of $d\times d$ symmetric matrix-valued broken polynomial fields.
The $L^2$-orthogonal projector on $\Poly{l}(\Th)$ is $\lproj{l}:L^1(\Omega)\to\Poly{l}(\Th)$ such that, for all $v\in L^1(\Omega)$,
\[
  (\lproj{l}v)_{|T}=\lproj[T]{l}v_{|T}\qquad\forall T\in\Th.
\]
The vector version $\vlproj{l}$ of the global $L^2$-orthogonal projector is obtained applying $\lproj{l}$ component-wise.

Broken polynomial spaces constitute special instances of the broken Sobolev spaces
\[
H^m(\Th)\coloneq\left\{v\in L^2(\Omega)\st v_{|T}\in H^m(T)\quad\forall T\in\Th\right\},
\]
which will be used to express the regularity requirements on the exact solution in the error estimates.
Hereafter, we denote by $\GRADh:H^1(\Th)\to L^2(\Omega)^d$ the broken gradient operator such that, for any $v\in H^1(\Th)$, $(\GRADh v)_{|T}\coloneq\GRAD v_{|T}$ for all $T\in\Th$.
Functions in $H^1(\Th)$ may exhibit jumps across interfaces.
To account for this fact, for all $F\in\Fhi$ with neighbouring elements $T_1$ and $T_2$ (remember that the ordering is fixed so that $\normal_F$ points out of $T_1$), we introduce the jump operator such that, for all $\varphi\in H^1(\Th)$,
\[
\jump{\varphi}\coloneq\left(\varphi_{|T_1}\right)_{|F} - \left(\varphi_{|T_2}\right)_{|F}.
\]
When applied to vector-valued functions, the jump operator acts element-wise.


\section{Hybrid High-Order discretisations}\label{sec:hho}

Throughout the rest of this section, $k\ge 0$ denotes the polynomial degree for the space approximation and the inequality $a\lesssim b$ stands for $a\le Cb$ with $C>0$ generic constant whose dependencies will be specified at each occurrence.

\subsection{Mechanical term}

The discretisation of the mechanical term hinges on the methods of \cite[Chapter 7]{Di-Pietro.Droniou:20} and \cite{Botti.Di-Pietro.ea:19*3}.

\subsubsection{Discrete displacement unknowns}

The space of discrete displacement unknowns is
\begin{equation*} 
  \uvec{U}_h^k\coloneq\left\{
  \uvec{v}_h\coloneq( (\bvec{v}_T)_{T\in\Th}, (\bvec{v}_F)_{F\in\Fh} )\st
  \bvec{v}_T\in\Poly{k}(T)^d\quad\forall T\in\Th\mbox{ and }
  \bvec{v}_F\in\Poly{k}(F)^d\quad\forall F\in\Fh
  \right\}.
\end{equation*}
For all $\uvec{v}_h\in\uvec{U}_h^k$, we let $\bvec{v}_h\in\Poly{k}(\Th)^d$ be the function obtained patching element unknowns, that is,
\begin{equation}\label{eq:vh}
  (\bvec{v}_h)_{|T}\coloneq\bvec{v}_T\qquad\forall T\in\Th.
\end{equation}
For any mesh element $T\in\Th$, we denote by $\uvec{U}_T^k$ and $\uvec{v}_T$ the restrictions to $T$ of $\uvec{U}_h^k$ and $\uvec{v}_h\in\uvec{U}_h^k$, respectively.
We also introduce the following subspace of $\uvec{U}_h^k$ that strongly accounts for the boundary condition \eqref{eq:biot:strong:bc.u}:
\begin{equation}\label{eq:Uhk0}
  \uvec{U}_{h,0}^k\coloneq\left\{
  \uvec{v}_h\in\uvec{U}_h^k\st \bvec{v}_F=\bvec{0}\quad\forall F\in\Fhb.
  \right\}.
\end{equation}
The global displacement interpolator $\uvec{I}_h^k:H^1(\Omega)^d\to\uvec{U}_h^k$ is such that, for all $\bvec{v}\in H^1(\Omega)^d$,
\[
\uvec{I}_h^k\bvec{v}\coloneq (
(\vlproj[T]{k}\bvec{v})_{T\in\Th}, (\vlproj[F]{k}\bvec{v})_{F\in\Fh}
).
\]

\subsubsection{Local reconstructions}

For any mesh element $T\in\Th$, we define the strain reconstruction $\ET:\uvec{U}_T^k\to\Polys{k}(T)$ such that it holds, for all $\uvec{v}_T\in\uvec{U}_T^k$,
\begin{equation}\label{eq:ET}
  (\ET\uvec{v}_T,\btens{\tau})_{L^2(T)^{d\times d}}
  = -(\bvec{v}_T,\DIV\btens{\tau})_{L^2(T)^d}
  + \sum_{F\in\Fh[T]}(\bvec{v}_F,\btens{\tau}\normal_{TF})_{L^2(F)^d}\qquad
  \forall\btens{\tau}\in\Polys{k}(T).
\end{equation}
We will also need the displacement reconstruction $\rT:\uvec{U}_T^k\to\Poly{k+1}(T)^d$ such that, for all $\uvec{v}_T\in\uvec{U}_T^k$,
\begin{gather*}
  (\GRADs\rT\uvec{v}_T-\ET\uvec{v}_T,\GRADs\bvec{w})_{L^2(T)^{d\times d}}=0\qquad\forall\bvec{w}\in\Poly{k+1}(T)^d,
  \\
  \int_T\rT\uvec{v}_T=\int_T\bvec{v}_T,\mbox{ and }
  \int_T\GRADss\rT\uvec{v}_T=\frac12\sum_{F\in\Fh[T]}\int_F\left(
  \bvec{v}_F\otimes\normal_{TF}-\normal_{TF}\otimes\bvec{v}_F
  \right),
\end{gather*}
where $\GRADss$ denotes the skew-symmetric part of the gradient.
The global strain and displacement reconstructions $\Eh:\uvec{U}_h^k\to\Polys{k}(\Th)$ and $\rh:\uvec{U}_h^k\to\Poly{k+1}(\Th)^d$ are obtained setting, for all $\uvec{v}_h\in\uvec{U}_h^k$,
\[
(\Eh\uvec{v}_h)_{|T}\coloneq\ET\uvec{v}_T,\qquad
(\rh\uvec{v}_h)_{|T}\coloneq\rT\uvec{v}_T\qquad\forall T\in\Th.
\]

\subsubsection{Discrete bilinear form and strain norm}

The bilinear form discretising the mechanical term is $\mathrm{a}_h:\uvec{U}_h^k\times\uvec{U}_h^k\to\Real$ such that, for all $\uvec{w}_h,\uvec{v}_h\in\uvec{U}_h^k$,
\begin{equation}\label{eq:ah}
  \mathrm{a}_h(\uvec{w}_h,\uvec{v}_h)
  \coloneq\begin{cases}
  (\btens{\sigma}(\Eh\uvec{w}_h),\Eh\uvec{v}_h)_{L^2(\Omega)^{d\times d}}
  + \mathrm{s}_{\mu,h}(\uvec{w}_h,\uvec{v}_h)
  & \text{if $k\ge 1$},
  \\[5pt]
  (\btens{\sigma}(\Eh[0]\uvec{w}_h),\Eh[0]\uvec{v}_h)_{L^2(\Omega)^{d\times d}}
  + \mathrm{s}_{\mu,h}(\uvec{w}_h,\uvec{v}_h)
  + \mathrm{j}_{\mu,h}(\rh[1]\uvec{w}_h,\rh[1]\uvec{v}_h)
  & \text{if $k=0$}.
  \end{cases}
\end{equation}
In the above expression, letting, for all $T\in\Th$, all $\uvec{v}_T\in\uvec{U}_T^k$, and all $F\in\Fh[T]$, $\bvec{\Delta}_{TF}^k\uvec{v}_T\coloneq\vlproj[F]{k}\left(\rT\uvec{v}_T-\bvec{v}_F\right)-\vlproj[T]{k}\left(\rT\uvec{v}_T-\bvec{v}_T\right)$, and denoting by $h_F$ the diameter of $F$, the stabilisation bilinear form $\mathrm{s}_{\mu,h}:\uvec{U}_h^k\times\uvec{U}_h^k\to\Real$ is such that, for all $\uvec{w}_h,\uvec{v}_h\in\uvec{U}_h^k$,
\[
\mathrm{s}_{\mu,h}(\uvec{w}_h,\uvec{v}_h)
\coloneq2\mu\sum_{T\in\Th}\sum_{F\in\Fh[T]}h_F^{-1}(\bvec{\Delta}_{TF}^k\uvec{w}_T,\bvec{\Delta}_{TF}^k\uvec{v}_T)_{L^2(F)^d},
\]
while the jump penalisation bilinear form $\mathrm{j}_{\mu,h}:H^1(\Th)^d\times H^1(\Th)^d\to\Real$ is such that, for all $\bvec{w},\bvec{v}\in H^1(\Th)^d$,
\[
\mathrm{j}_{\mu,h}(\bvec{w},\bvec{v})
\coloneq 2\mu\sum_{F\in\Fh}h_F^{-1} (\jump{\bvec{w}},\jump{\bvec{v}})_{L^2(F)^d},
\]
where we have extended the definition of the jump operator to boundary faces setting $\jump{\varphi}\coloneq\varphi$ for all $F\in\Fhb$.
This jump penalisation bilinear form is required to ensure stability when $k=0$, cf. the discussion in \cite[Section 4.4]{Botti.Di-Pietro.ea:19*3}.

Define on $\uvec{U}_h^k$ the strain seminorm such that, for all $\uvec{v}_h\in\uvec{U}_h^k$, 
\begin{multline}\label{eq:norm.strain.h}
  \norm[\strain,h]{\uvec{v}_h}^2\coloneq\\
  \begin{cases}
  \norm[L^2(\Omega)^{d\times d}]{\GRADsh\bvec{v}_h}^2 + \displaystyle\sum_{T\in\Th}\sum_{F\in\Fh[T]}h_F^{-1}\norm[L^2(F)^d]{\bvec{v}_F-\bvec{v}_T}^2
  & \text{if $k\ge 1$},
  \\
  \norm[L^2(\Omega)^{d\times d}]{\GRADsh\rh[1]\uvec{v}_h}^2
  + \displaystyle\sum_{T\in\Th}\displaystyle\sum_{F\in\Fh[T]} h_F^{-1}\norm[L^2(F)^d]{\bvec{\Delta}_{TF}^0\uvec{v}_T}^2
  + \displaystyle\sum_{F\in\Fh}h_F^{-1}\norm[L^2(F)^d]{\jump{\rh[1]\uvec{v}_h}}^2
  & \text{if $k=0$}.
  \end{cases}
\end{multline}
The map $\norm[\strain,h]{{\cdot}}$ is a norm on $\uvec{U}_{h,0}^k$. 
For future use, we note the following boundedness property (see \cite[Lemmas 7.12]{Di-Pietro.Droniou:20} for a detailed proof in the case $k\ge 1$): For all $\bvec{v}\in H^1(\Omega)^d$,
\begin{equation}\label{eq:Ih.boundedness}
  \norm[\strain,h]{\uvec{I}_h^k\bvec{v}}\lesssim\norm[H^1(\Omega)^d]{\bvec{v}},
\end{equation}
where the hidden constant depends only on $d$, $k$, and the mesh regularity parameter.

The stability and boundedness of $\mathrm{a}_h$ are expressed by the existence of a real number $C_{\mathrm{a}}>0$ independent of $h$, $\mu$, and $\lambda$ such that, for all $\uvec{v}_h\in\uvec{U}_h^k$,
\begin{equation}\label{eq:ah:coercivity.boundedness}
  C_{\mathrm{a}}^{-1}2\mu\norm[\strain,h]{\uvec{v}_h}^2
  \le\mathrm{a}_h(\uvec{v}_h,\uvec{v}_h)
  \le C_{\mathrm{a}}\left(2\mu+d\lambda^+\right)\norm[\strain,h]{\uvec{v}_h}^2.
\end{equation}
The proof is a straightforward adaptation of \cite[Lemmas 7.27]{Di-Pietro.Droniou:20} for the case $k\ge 1$ and to \cite[Lemma 8]{Botti.Di-Pietro.ea:19*3} for the case $k=0$ (in these lemmas, $\lambda^+$ is replaced by $|\lambda|$ in the right-hand side).
Moreover it holds, for all $\bvec{w}\in H_0^1(\Omega)^d$ such that $\bvec{w}\in H^{m+2}(\Th)^d$ for some $m\in\{0,\ldots,k\}$ and $\DIV\btens{\sigma}(\GRADs\bvec{w})\in L^2(\Omega)^d$,
\begin{equation}\label{eq:ah:consistency}
  \norm[\strain,h,*]{\mathcal{E}_{\mathrm{a},h}(\bvec{w};\cdot)}
  \lesssim h^{m+1}\left(
  \seminorm[H^{m+1}(\Th)^{d\times d}]{\btens{\sigma}(\GRADs\bvec{w})}
  + 2\mu\seminorm[H^{m+2}(\Th)^d]{\bvec{w}}
  \right),
\end{equation}
where $\norm[\strain,h,*]{{\cdot}}$ denotes the norm dual to $\norm[\strain,h]{{\cdot}}$ such that, for any linear form $\ell_h:\uvec{U}_{h,0}^k\to\Real$,
\[
\norm[\strain,h,*]{\ell_h}\coloneq\sup_{\uvec{v}_h\in\uvec{U}_{h,0}^k\setminus\{\uvec{0}\}}\frac{\ell_h(\uvec{v}_h)}{\norm[\strain,h]{\uvec{v}_h}},
\]
and the mechanical consistency error linear form is such that, for all $\uvec{v}_h\in\uvec{U}_{h,0}^k$,
\begin{equation}\label{eq:ah:consistency.error}
  \mathcal{E}_{\mathrm{a},h}(\bvec{w};\uvec{v}_h)
  \coloneq -(\DIV\btens{\sigma}(\GRADs\bvec{w}),\bvec{v}_h)_{L^2(\Omega)^d} - \mathrm{a}_h(\bvec{I}_h^k\bvec{w},\uvec{v}_h).
\end{equation}

\subsection{Hybrid High-Order discretisation of the Darcy term}

We next describe the HHO discretisation of the Darcy term based on the method of \cite[Section 3]{Di-Pietro.Ern.ea:16}; see also \cite[Section 3.1]{Di-Pietro.Droniou:20} for a detailed analysis.

\subsubsection{Discrete pore pressure unknowns}

We define the space of discrete pore pressure unknowns
\begin{equation}\label{eq:Phk.hho}
  \underline{P}_h^k\coloneq\left\{
  \underline{q}_h\coloneq( (q_T)_{T\in\Th}, (q_F)_{F\in\Fh} )\st
  q_T\in\Poly{k}(T)\quad\forall T\in\Th\mbox{ and }
  q_F\in\Poly{k}(F)\quad\forall F\in\Fh
  \right\}
\end{equation}
and, for all $\underline{q}_h\in\underline{P}_h^k$, we let $q_h\in\Poly{k}(\Th)$ be the broken polynomial function such that
\begin{equation}\label{eq:qh}
  (q_h)_{|T}\coloneq q_T\qquad\forall T\in\Th.
\end{equation}
For any mesh element $T\in\Th$, we denote by $\underline{P}_T^k$ and $\underline{q}_T$ the restrictions to $T$ of $\underline{P}_h^k$ and $\underline{q}_h\in\underline{P}_h^k$, respectively.
We also define the global pressure interpolator $\underline{I}_h^k:H^1(\Omega)\to\underline{P}_h^k$ such that, for all $q\in H^1(\Omega)$,
\[
\underline{I}_h^k q\coloneq ( (\lproj[T]{k} q)_{T\in\Th}, (\lproj[F]{k} q)_{F\in\Fh} ).
\]

\subsubsection{Local reconstruction}

For any mesh element $T\in\Th$, we reconstruct a high-order pore pressure through the operator $\pT:\underline{P}_T^k\to\Poly{k+1}(T)$ such that, for all $\underline{q}_T\in\underline{P}_T^k$,
\begin{gather*}
  (\diff_T\GRAD\pT\underline{q}_T,\GRAD r)_{L^2(T)^d}
  = -(q_T,\DIV(\diff_T\GRAD r))_{L^2(T)^d}
  + \sum_{F\in\Fh[T]} (q_F,\diff_T\GRAD r\cdot\normal_{TF})_{L^2(F)}\qquad
  \forall r\in\Poly{k+1}(T)
  \\
  \int_\Omega\pT\underline{q}_T=\int_\Omega q_T.
\end{gather*}
A global pore pressure reconstruction $\ph:\underline{P}_h^k\to\Poly{k+1}(\Th)$ is obtained patching the local reconstructions:
For all $\underline{q}_h\in\underline{P}_h^k$,
\begin{equation}\label{eq:ph}
  (\ph\underline{q}_h)_{|T}\coloneq\pT\underline{q}_T\qquad\forall T\in\Th.
\end{equation}

\subsubsection{Bilinear form}

The bilinear form for the HHO discretisation of the Darcy term is $\mathrm{c}_h^\hho:\underline{P}_h^k\times\underline{P}_h^k\to\Real$ such that, for all $\underline{r}_h,\underline{q}_h\in\underline{P}_h^k$,
\begin{equation}\label{eq:ch.hho}
  \mathrm{c}_h^\hho(\underline{r}_h,\underline{q}_h)
  \coloneq
  (\diff\GRADh\ph\underline{r}_h,\GRADh\ph\underline{q}_h)_{L^2(\Omega)^d}
  + \mathrm{s}_{\diff,h}(\underline{r}_h,\underline{q}_h),
\end{equation}
where the stabilisation bilinear form $\mathrm{s}_{\diff,h}$ is such that
\[
\mathrm{s}_{\diff,h}(\underline{r}_h,\underline{q}_h)
\coloneq\sum_{T\in\Th}\sum_{F\in\Fh[T]}\frac{\kappa_{TF}}{h_F}(\Delta_{TF}^k\underline{r}_T,\Delta_{TF}^k\underline{q}_T)_{L^2(F)}
\]
with $\kappa_{TF}\coloneq\diff_T\normal_{TF}\cdot\normal_{TF}$ and, for all $T\in\Th$, all $\underline{q}_T\in\underline{P}_T^k$, and all $F\in\Fh[T]$, $\Delta_{TF}^k\underline{q}_T\coloneq\lproj[F]{k}(\pT\underline{q}_T-q_F) - \lproj[T]{k}(\pT\underline{q}_T-q_T)$.
It can be checked that the bilinear form $\mathrm{c}_h$ is coercive on
\[
\underline{P}_{h,0}^k\coloneq\left\{\underline{q}_h\in\underline{P}_h^k\st\int_\Omega q_h=0\right\},
\]
hence, being symmetric, it defines an inner-product and an induced norm on this space.
Set, for any $\underline{q}_h\in\underline{P}_h^k$,
\begin{equation}\label{eq:norm.hho.h}
  \norm[\hho,h]{\underline{q}_h}\coloneq\mathrm{c}_h^\hho(\underline{q}_h,\underline{q}_h)^{\frac12}
\end{equation}
and denote by $\norm[\hho,h,*]{{\cdot}}$ the corresponding dual norm such that, for any linear form $\ell_h:\underline{P}_{h,0}^k\to\Real$,
\[
\norm[\hho,h,*]{\ell_h}\coloneq\sup_{\underline{q}_h\in\underline{P}_{h,0}^k\setminus\{0\}}\frac{\ell_h(\underline{q}_h)}{\norm[\hho,h]{\underline{q}_h}}.
\]

Consider now the HHO pore pressure projection problem:
Given $r\in H^1(\Omega)$ such that $\DIV(\diff\GRAD r)\in L^2(\Omega)$, find $\underline{r}_h\in\underline{P}_h^k$ such that
\begin{equation}\label{eq:ch.hho.projection}
  \begin{aligned}
    \mathrm{c}_h^\hho(\underline{r}_h,\underline{q}_h) &= -(\DIV(\diff\GRAD r),q_h)_{L^2(\Omega)}\qquad\forall\underline{q}_h\in\underline{P}_h^k,
    \\
    \int_\Omega r_h(\bvec{x})\ud\bvec{x} &= \int_\Omega r(\bvec{x})\ud\bvec{x}.
  \end{aligned}
\end{equation}

\begin{proposition}[Error estimate for the pore pressure projection problem]
  Let $r\in H^1(\Omega)\cap H^{m+2}(\Th)$ for some $m\in\{0,\ldots,k\}$ be such that $\DIV(\diff\GRAD r)\in L^2(\Omega)$, and denote by $\underline{r}_h\in\underline{P}_h^k$ the unique solution to problem \eqref{eq:ch.hho.projection}.
  Then, it holds that
  \begin{equation}\label{eq:ch.hho.projection:err.est}
    \norm[L^2(\Omega)]{r_h-r}
    + h \norm[L^2(\Omega)^d]{\GRADh(r_h-r)}
    \lesssim\rho\left(
      \sum_{T\in\Th}\rho_Th_T^{2(m+1)}\seminorm[H^{m+2}(T)]{r}^2
      +  h^{2(m+1)}\seminorm[H^{m+1}(\Th)]{r}^2
    \right)^{\frac12},
  \end{equation}
  with hidden constant independent of $h$, $\diff$, and $r$.
\end{proposition}
\begin{proof}
  Let, for any $\underline{q}_h\in\underline{P}_h^k$, $\seminorm[1,h]{\underline{q}_h}^2\coloneq\sum_{T\in\Th}\sum_{F\in\Fh[T]}h_F^{-1}\norm[L^2(F)]{q_F-q_T}^2$.
  A straightforward adaptation of the estimate of \cite[Theorem 3.18]{Di-Pietro.Droniou:20} to the case of homogeneous Neumann boundary conditions combined with the seminorm equivalence of \cite[Lemma 3.15, Point (i)]{Di-Pietro.Droniou:20} gives
  \[
  \norm[L^2(\Omega)^d]{\GRADh(r_h-\lproj{k} r)}
  + \seminorm[1,h]{\underline{r}_h-\underline{I}_h^k r}
  \lesssim \rho\left(
  \sum_{T\in\Th}\rho_Th_T^{2(m+1)}\seminorm[H^{m+2}(T)]{r}^2
  \right)^{\frac12}.
  \]
  By the discrete Poincar\'e--Wirtinger inequality of \cite[Theorem 6.5]{Di-Pietro.Droniou:20} with $p=q=2$ (which can be applied here since the function $r_h-\lproj{k} r$ has zero-mean value over $\Omega$), $\norm[L^2(\Omega)]{r_h-\lproj{k}r}$ is bounded by the left-hand side of the above expression with constant independent of both $h$ and $\diff$, hence the same bound holds also for this quantity.
  Inserting $\lproj{k}r$ into the norms in the left-hand side of \eqref{eq:ch.hho.projection:err.est} and using the triangle inequality, we next infer
  \[
  \begin{aligned}
    \norm[L^2(\Omega)]{r_h-r} + h \norm[L^2(\Omega)^d]{\GRADh(r_h-r)}
    &\le
    \norm[L^2(\Omega)]{r_h-\lproj{k}r} + h \norm[L^2(\Omega)^d]{\GRADh(r_h-\lproj{k}r)}
    \\
    &\quad
    + \norm[L^2(\Omega)]{\lproj{k}r -r} + h \norm[L^2(\Omega)^d]{\GRADh(\lproj{k}r-r)}.
  \end{aligned}
  \]
  To conclude, apply the estimates discussed above to the terms in the first line, use the approximation properties of the $L^2$-orthogonal projector (see \cite{Di-Pietro.Droniou:17}) for the terms in the second line, and observe that, denoting by $h_\Omega$ the diameter of $\Omega$, $h\le h_\Omega\lesssim 1$.
\end{proof}
\begin{remark}[Improved error estimates]
  When elliptic regularity holds, an improved error estimate in $h^{m+2}$ can be proved for the $L^2$-norm of the difference between the pressure and its global reconstruction obtained through $\ph$ (see \eqref{eq:ph}); see \cite[Remark 3.21]{Di-Pietro.Droniou:20}.
  The dependence on the global anisotropy ratio can also be avoided when estimating the error in the natural energy norm; see \cite[Theorems 3.18 and 3.19]{Di-Pietro.Droniou:20}.
  The estimate \eqref{eq:ch.hho.projection:err.est} is, however, the one required for the following analysis.
\end{remark}

\subsection{Discontinuous Galerkin discretisation of the Darcy term}\label{sec:dg.darcy}

Assume $k\ge 1$ and define the weighted average operator such that, for all $\varphi\in H^1(\Th)$ and all $F\in\Fhi$ shared by the mesh elements $T_1$ and $T_2$,
\[
\wavg{\varphi}\coloneq\omega_{F,1}\left(\varphi_{|T_1}\right)_{|F} + \omega_{F,2}\left(\varphi_{|T_2}\right)_{|F},
\]
where, letting $\kappa_{T_iF}\coloneq(\diff_{T_i}\normal_F\cdot\normal_F)^{\frac12}$, $i\in\{1,2\}$, the weights are such that
\[
\omega_1\coloneq\frac{\kappa_{T_2F}^{\nicefrac12}}{\kappa_{T_1F}^{\nicefrac12}+\kappa_{T_2F}^{\nicefrac12}}\mbox{ and }
\omega_2=1-\omega_1=\frac{\kappa_{T_1F}^{\nicefrac12}}{\kappa_{T_1F}^{\nicefrac12}+\kappa_{T_2F}^{\nicefrac12}}.
\]
When applied to vector-valued functions, the average operator acts component-wise.
Let a polynomial degree $k\ge 1$ be fixed.
The Discontinous Galerkin discretisation of the Darcy term hinges on the bilinear form $\mathrm{c}_h^\dg:\Poly{k}(\Th)\times\Poly{k}(\Th)\to\Real$ such that, for all $r_h,q_h\in\Poly{k}(\Th)$,
\begin{equation}\label{eq:ch.dg}
  \begin{aligned}
    \mathrm{c}_h^\dg(r_h,q_h)
    &\coloneq(\diff\GRADh r_h,\GRADh q_h)_{L^2(\Omega)^d}
    + \sum_{F\in\Fhi}\frac{\eta\kappa_F}{h_F} (\jump{r_h},\jump{q_h})_{L^2(F)}
    \\
    &\quad
    - \sum_{F\in\Fhi}\left(
    (\jump{r_h},\wavg{\diff\GRADh q_h}\cdot\normal_F)_{L^2(F)}
    + (\wavg{\diff\GRADh r_h}\cdot\normal_F,\jump{q_h})_{L^2(F)}
    \right),
  \end{aligned}
\end{equation}
where, for any interface $F\in\Fhi$ shared by the elements $T_1,T_2\in\Th$, $\kappa_F\coloneq\min(\kappa_{T_1F},\kappa_{T_2F})$, and $\eta>0$ is a stabilisation parameter.
Furnish $\Poly{k}(\Th)\cap L^2_0(\Omega)$ with the norm such that, for all $q_h\in\Poly{k}(\Th)$,
\begin{equation}\label{eq:norm.dg.h}
  \norm[\dg,h]{q_h}\coloneq 
  \left(
  \norm[L^2(\Omega)^d]{\GRADh q_h}^2 + \sum_{F\in\Fhi}\frac{\kappa_F}{h_F}\norm[L^2(F)]{\jump{q_h}}^2
  \right)^{\frac12}.
\end{equation}
Provided that $\eta$ is large enough, the bilinear form $\mathrm{c}_h^\dg$ is coercive with respect to this norm with coercivity constant $\gamma^\dg$ depending only on $\eta$, a discrete trace constant, and an upper bound on number of faces of one mesh element.

Let $r\in H^1(\Omega)$ be such that $\DIV(\diff\GRAD r)\in L^2(\Omega)$, and consider the problem that consists in finding $r_h\in\Poly{k}(\Th)$ such that
\begin{equation}\label{eq:ch.dg.projection}
  \begin{aligned}
    \mathrm{c}_h^\dg(r_h,q_h) &= -(\DIV(\diff\GRAD r), q_h)_{L^2(\Omega)} \qquad\forall q_h\in\Poly{k}(\Th),
    \\
    \int_\Omega r_h(\bvec{x})\ud\bvec{x} &= \int_\Omega r(\bvec{x})\ud\bvec{x}.
  \end{aligned}
\end{equation}
Then, assuming elliptic regularity (that is, $\Omega$ convex and $\diff$ constant over $\Omega$) and  $r\in H^{m+1}(\Th)$ for some $m\in\{0,\ldots,k\}$, it holds
\begin{equation}\label{eq:ch.dg.projection:err.est}
  \underline{K}^{\frac12}\norm[L^2(\Omega)]{r_h-r} + h \norm[\dg,h]{r_h-r}\lesssim h^{m+1}\overline{K}^{\frac12}\seminorm[H^{m+1}(\Th)]{r},
\end{equation}
with hidden constant independent of $h$, $\diff$, and $r$.
The proof of this result, that requires the assumption $k\ge 1$, is a straightforward adaptation of \cite[Appendix A]{Di-Pietro.Droniou:18*1} to the case of homogeneous Neumann boundary conditions.

\begin{remark}[Role of elliptic regularity]
  Without elliptic regularity, the first term in the left-hand side of \eqref{eq:ch.dg.projection:err.est} is only $\mathcal{O}(h^m)$, and one order of convergence is lost in estimate \eqref{eq:err.est_dg} for the Biot problem. Notice that this restriction is not present when using the HHO discretisation of the Darcy term.
\end{remark}

\subsection{Displacement--pore pressure coupling term}

Define the global discrete divergence operator $\Dh:\uvec{U}_h^k\to\Poly{k}(\Th)$ such that, for all $\uvec{v}_h\in\uvec{U}_h^k$,
\[
\Dh\uvec{v}_h\coloneq\tr(\Eh\uvec{v}_h),
\]
where $\tr$ denotes the trace operator.
By definition, we have the following commutation property: For all $\bvec{v}\in H^1(\Omega)^d$,
\begin{equation}\label{eq:Dh.commutation}
  \Dh(\uvec{I}_h^k\bvec{v}) = \lproj{k}(\DIV\bvec{v}).
\end{equation}
The displacement-pore pressure coupling is realised by the bilinear form $\mathrm{b}_h:\uvec{U}_h^k\times\Poly{k}(\Th)\to\Real$ such that, for any $(\uvec{v}_h,q_h)\in\uvec{U}_h^k\times\Poly{k}(\Th)$,
\begin{equation}\label{def:bh}
  \mathrm{b}_h(\uvec{v}_h,q_h)\coloneq - (\Dh\uvec{v}_h, q_h)_{L^2(\Omega)}.
\end{equation}
Expanding $\Dh$ according to its definition and invoking, inside each element $T\in\Th$, \eqref{eq:ET} with $\btens{\tau}=q_T\Id$ where $q_T\coloneq q_{h|T}$, we obtain the following equivalent expression:
\begin{equation}\label{eq:bh}
\mathrm{b}_h(\uvec{v}_h,q_h)
= - \sum_{T\in\Th}\left[
  (\DIV\bvec{v}_T, q_T)_{L^2(T)}
  + \sum_{F\in\Fh[T]}(\bvec{v}_F-\bvec{v}_T,q_T\normal_{TF})_{L^2(F)^d}
  \right].
\end{equation}
For use in the following lemma, we notice that, by \cite[Lemmas 7.24 and 7.42]{Di-Pietro.Droniou:20}, we have the following discrete Korn inequality:
\begin{equation}\label{eq:korn.discrete}
  \norm[1,h]{\uvec{v}_h}\coloneq\left(
  \norm[L^2(\Omega)^{d\times d}]{\GRADh\bvec{v}_h}^2
  + \sum_{T\in\Th}\sum_{F\in\Fh[T]}h_F^{-1}\norm[L^2(F)^d]{\bvec{v}_F-\bvec{v}_T}^2
  \right)^{\frac12}\lesssim\norm[\strain,h]{\uvec{v}_h}\qquad
  \forall\uvec{v}_h\in\uvec{U}_{h,0}^k,
\end{equation}
where the hidden constant depends only on $\Omega$, $d$, $k$, and the mesh regularity parameter.
\begin{lemma}[Properties of the bilinear form $\mathrm{b}_h$]
  The bilinear form $\mathrm{b}_h$ has the following properties:
  \begin{enumerate}
  \item \emph{Consistency/1.} For all $\bvec{v}\in H^1(\Omega)^d$ it holds, with bilinear form $b$ defined by \eqref{eq:a.b.c},
    \begin{equation}\label{eq:bh:consistency:1}
      \mathrm{b}_h(\uvec{I}_h^k\bvec{v},q_h) = b(\bvec{v},q_h)\qquad\forall q_h\in\Poly{k}(\Th).
    \end{equation}
  \item \emph{Inf-sup stability.} There is a real number $\beta>0$ independent of $h$, but possibly depending on $\Omega$, $d$, $k$, and the mesh regularity parameter such that
    \begin{equation}\label{eq:bh:inf-sup}
      \forall q_h\in \Poly{k}(\Th),\qquad
      \beta\norm[L^2(\Omega)]{q_h-\lproj[\Omega]{0}q_h}\le\norm[\strain,h,*]{\mathrm{b}_h(\cdot,q_h)}.
    \end{equation}
  \item \emph{Consistency/2.} Let $r\in H^1(\Omega)\cap H^{m+2}(\Th)$ for some $m\in\{0,\ldots,k\}$  be such that $\DIV(\diff\GRAD r)\in L^2(\Omega)$.
    Let $\underline{r}_h^\hho\in\underline{P}_h^k$ be the solution to problem \eqref{eq:ch.hho.projection}.
    Then, it holds that
    \begin{equation}\label{eq:bh:consistency:hho}
      \norm[\strain,h,*]{\mathcal{E}_{\mathrm{b},h}^\hho(r;\cdot)}\lesssim
      \rho\left(
      \sum_{T\in\Th}\rho_Th_T^{2(m+1)}\seminorm[H^{m+2}(T)]{r}^2 + h^{2(m+1)}\seminorm[H^{m+1}(\Th)]{r}^2
      \right)^{\frac12},
    \end{equation}
    where the consistency error linear form $\mathcal{E}_{\mathrm{b},h}^\hho(r;\cdot):\uvec{U}_h^k\to\Real$ is such that, for all $\uvec{v}_h\in\uvec{U}_h^k$,
    \begin{equation}\label{eq:bh.hho:consistency.error}
      \mathcal{E}_{\mathrm{b},h}^\hho(r;\uvec{v}_h)
      \coloneq (\bvec{v}_h,\GRAD r)_{L^2(\Omega)^d} - \mathrm{b}_h(\uvec{v}_h,r_h^\hho).
    \end{equation}
  \item \emph{Consistency/3.} Assume $k\ge 1$. Let $r\in H^1(\Omega)\cap H^{m+1}(\Th)$ for some $m\in\{0,\ldots,k\}$  be such that $\DIV(\diff\GRAD r)\in L^2(\Omega)$.
    Let $r_h^\dg\in\Poly{k}(\Th)$ be the solution to problem \eqref{eq:ch.dg.projection}.
    Then, assuming elliptic regularity, it holds that
    \begin{equation}\label{eq:bh:consistency:dg}
      \norm[\strain,h,*]{\mathcal{E}_{\mathrm{b},h}^\dg(r;\cdot)}\lesssim
      \rho^{\frac12} h_T^{m+1}\seminorm[H^{m+1}(\Th)]{r},
    \end{equation}
    where the consistency error linear form $\mathcal{E}_{\mathrm{b},h}^\dg(r;\cdot):\uvec{U}_h^k\to\Real$ is such that, for all $\uvec{v}_h\in\uvec{U}_h^k$,
    \begin{equation}\label{eq:bh.dg:consistency.error}
      \mathcal{E}_{\mathrm{b},h}^\dg(r;\uvec{v}_h)
      \coloneq (\bvec{v}_h,\GRAD r)_{L^2(\Omega)^d} - \mathrm{b}_h(\uvec{v}_h,r_h^\dg).
    \end{equation}
  \item \emph{Boundedness.} It holds, for all $(\uvec{v}_h,q_h)\in\uvec{U}_h^k\times\Poly{k}(\Th)$,
    \[
      |\mathrm{b}_h(\uvec{v}_h,q_h)|
      \lesssim\norm[\strain,h]{\uvec{v}_h}\norm[L^2(\Omega)]{q_h},
    \]
    where the hidden constant depends only on $d$, $k$, and the mesh regularity parameter.
  \end{enumerate}
\end{lemma}
\begin{proof}
  \underline{1. \emph{Consistency/1.}}
  Recalling the definition \eqref{def:bh} of $\mathrm{b}_h$ and the commutation property \eqref{eq:Dh.commutation} of the discrete divergence we infer, for all $q_h\in\Poly{k}(\Th)$, $\mathrm{b}_h(\uvec{I}_h^k\bvec{v},q_h)=-(\lproj{k}(\DIV\bvec{v}),q_h)_{L^2(\Omega)}=-(\DIV\bvec{v},q_h)_{L^2(\Omega)}=b(\bvec{v},q_h)$, where we have removed $\lproj{k}$ in the second passage by invoking its definition.
  \medskip\\
  \underline{2. \emph{Inf-sup stability.}}
  Combine the boundedness \eqref{eq:Ih.boundedness} of the interpolator with \eqref{eq:bh:consistency:1} and use Fortin's argument (see, e.g., \cite{Fortin:77} and also \cite[Section 5.4]{Boffi.Brezzi.ea:13}) after observing that $\mathrm{b}_h(\uvec{v}_h,\lproj[\Omega]{0}q_h)=0$ for all $\uvec{v}_h\in\uvec{U}_{h,0}^k$.
  \medskip\\
  \underline{3. \emph{Consistency/2.}}
  Let $\uvec{v}_h\in\bvec{U}_{h,0}^k$ be such that $\norm[\strain,h]{\uvec{v}_h}=1$.
  Integrating by parts element by element the first term in the right-hand side of \eqref{eq:bh.hho:consistency.error} and using the continuity of $r$ across interfaces together with the fact that displacement unknowns vanish on boundary faces to insert $\bvec{v}_F$ into the second term, we obtain
  \[
  (\bvec{v}_h,\GRAD r)_{L^2(\Omega)^d}
  = -\sum_{T\in\Th}\left[
    (\DIV\bvec{v}_T,r)_{L^2(T)} + \sum_{F\in\Fh[T]}(\bvec{v}_F-\bvec{v}_T,r\normal_{TF})_{L^2(F)^d}
  \right].
  \]
  Subtracting from this quantity \eqref{eq:bh} written for $q_h=r_h^\hho$, we obtain
  \begin{equation}\label{eq:bh:consistency.2-3:basic}
    \begin{aligned}
      &\mathcal{E}_{\mathrm{b},h}^\hho(r;\uvec{v}_h)
      \\
      &\quad
      =\sum_{T\in\Th}\left[
        (\DIV\bvec{v}_T,r_T^\hho-r)_{L^2(T)}
        + \sum_{F\in\Fh[T]}(\bvec{v}_F-\bvec{v}_T,(r_T^\hho-r)\normal_{TF})_{L^2(F)^d}
        \right]
      \\
      &\quad
      \le\sum_{T\in\Th}\left[
        \norm[L^2(T)]{\DIV\bvec{v}_T}\norm[L^2(T)]{r_T^\hho-r}
        + \sum_{F\in\Fh[T]}h_F^{-\frac12}\norm[L^2(F)^d]{\bvec{v}_F-\bvec{v}_T}~h_T^{\frac12}\norm[L^2(F)]{r_T^\hho-r}
        \right]
      \\
      &\quad
      \le
      \norm[1,h]{\uvec{v}_h}\left(
      \norm[L^2(\Omega)]{r_h^\hho-r}^2
      + \sum_{T\in\Th} h_T\norm[L^2(\partial T)]{r_h^\hho-r}^2
      \right)^{\frac12}
      \\
      &\quad
      \lesssim
      \norm[\strain,h]{\uvec{v}_h}\left(
      \norm[L^2(\Omega)]{r_h^\hho-r}
      + h \norm[L^2(\Omega)^d]{\GRADh(r_h^\hho-r)}
      \right),
    \end{aligned}
  \end{equation}
  where we have used a Cauchy--Schwarz inequality for the first term and a generalised H\"older inequality with exponents $(2,2,\infty)$ along with $\norm[L^\infty(F)^d]{\normal_{TF}}=1$ for the second term in the first bound,
  Cauchy--Schwarz inequalities on the sums along with $h_F\le h_T$ for all $F\in\Fh[T]$ and the definition of the $\norm[1,h]{{\cdot}}$-norm in the second bound,
  and the discrete Korn inequality \eqref{eq:korn.discrete} along with a local continuous trace inequality (cf. \cite[Lemma 1.31]{Di-Pietro.Droniou:20}) to conclude.
  The estimate \eqref{eq:bh:consistency:hho} follows combining \eqref{eq:bh:consistency.2-3:basic} with \eqref{eq:ch.hho.projection:err.est}.
  \medskip\\
  \underline{4. \emph{Consistency/3.}}
  Apply \eqref{eq:ch.dg.projection:err.est} in lieu of \eqref{eq:ch.hho.projection:err.est} to \eqref{eq:bh:consistency.2-3:basic} with $r_h^\hho$ replaced by $r_h^\dg$.
  \medskip\\
  \underline{5. \emph{Boundedness.}}
  Starting from \eqref{eq:bh} and proceeding similarly to \eqref{eq:bh:consistency.2-3:basic}, we can write
  \[
  |\mathrm{b}_h(\uvec{v}_h,q_h)|
  \lesssim\norm[\strain,h]{\uvec{v}_h}\left(
  \norm[L^2(\Omega)]{q_T}^2 + \sum_{T\in\Th} h_T\norm[L^2(\partial T)]{q_T}^2
  \right)^{\frac12}
  \lesssim\norm[\strain,h]{\uvec{v}_h}\norm[L^2(\Omega)]{q_h},
  \]
  where the conclusion follows from the local discrete trace inequality of \cite[Lemma 1.32]{Di-Pietro.Droniou:20}.
\end{proof}


\subsection{Discrete problems}

We consider in this section two HHO schemes using, respectively, a DG and an HHO discretisation of the Darcy term.
For all $1\le n\le N$ and $V$ vector space, define the time interpolator $I_\tau^n:L^2(t^{n-1},t^n;V)\to V$ such that, for all $\varphi\in L^2(t^{n-1},t^n;V)$,
\begin{equation}\label{eq:I.tau.n}
  I_\tau^n\varphi\coloneq\frac1\tau\int_{t^{n-1}}^{t^n}\varphi(t)\ud t.
\end{equation}
In both cases, we consider the following approximations of the source terms, corresponding to a lowest-order DG method in time:
For all $1\le n\le N$,
\begin{equation}\label{eq:overline.fhn.ghn}
  \overline{\bvec{f}}^n\coloneq I_\tau^n\bvec{f},\qquad
  \overline{g}^n\coloneq I_\tau^n g.
\end{equation}
\begin{problem}[HHO-HHO scheme]\label{prob:hho-hho}
  Let $k\ge 0$.
  The families $\uvec{u}_{h\tau}\coloneq(\uvec{u}_h^n)_{0\le n\le N}\in(\uvec{U}_{h,0}^k)^{N+1}$ and $\underline{p}_{h\tau}\coloneq(\underline{p}_h^n)_{0\le n\le N}\in(\underline{P}_h^k)^{N+1}$ are such that, for $n=1,\ldots, N$,
  $$
    \begin{aligned}
      \mathrm{a}_h(\uvec{u}_h^n,\uvec{v}_h) + \mathrm{b}_h(\uvec{v}_h,p_h^n) 
      &= (\overline{\bvec{f}}^n,\bvec{v}_h)_{L^2(\Omega)^d}
      &\qquad& \forall\uvec{v}_h\in\uvec{U}_{h,0}^k,
      \\ 
      C_0(\delta_t^n p_{h\tau},q_h)_{L^2(\Omega)} - \mathrm{b}_h(\delta_t^n\uvec{u}_{h\tau},q_h) 
      + \mathrm{c}_h^\hho(\underline{p}_h^n,\underline{q}_h) &= (\overline{g}^n,q_h)_{L^2(\Omega)}
      &\qquad& \forall\underline{q}_h\in\underline{P}_h^k,
    \end{aligned}
  $$
  and, if $C_0=0$, for all $1\le n\le N$,
    \[
    \int_\Omega p_h^n(\bvec{x})\ud\bvec{x}=0,
    \]
    where $\bvec{v}_h$ is defined from $\uvec{v}_h$ according to \eqref{eq:vh}, while $p_h^n$ and $q_h$ are respectively 
    defined from $\underline{p}_h^n$ and $\underline{q}_h$ according to \eqref{eq:qh}.
    The initial values of the discrete displacement and pore pressure are chosen such that
    \[
    C_0(p_h^0,q_h)_{L^2(\Omega)} - \mathrm{b}_h(\uvec{u}_h^0,q_h) = (\phi^0,q_h)_{L^2(\Omega)}\qquad\forall q_h\in P_h^k.
    \]
\end{problem}
\begin{table}[h]
  \caption{Interpretation of the HHO-HHO method within the abstract framework of Section \ref{sec:abstract.framework}
           \label{tab:interpretation_hho}}
  \centering
  \begin{small}
    \renewcommand\arraystretch{1.2}
    \begin{tabular}{ccc}
      \toprule
      Abstract object & Instance & Definition
      \\
      \midrule
      $\bvec{\abs{U}}_h$, $\norm[\bvec{\abs{U}},h]{{\cdot}}$ & $\uvec{U}_h^k$, $\norm[\strain,h]{{\cdot}}$ & \eqref{eq:Uhk0}, \eqref{eq:norm.strain.h}
      \\
      $\abs{P}_h$, $\norm[\abs{P},h]{{\cdot}}$ & $\underline{P}_h^k$, $\norm[\hho,h]{{\cdot}}$ & \eqref{eq:Phk.hho}, \eqref{eq:norm.hho.h}
      \\
      $\abs{r}_h$ & $\underline{P}_h^k\ni\underline{q}_h\mapsto q_h\in\Poly{k}(\Th)\hookrightarrow L^2(\Omega)$ & \eqref{eq:qh}
      \\
      $\abs{L}_h$ & $\Poly{k}(\Th)$ & \eqref{eq:qh}      
      \\
      $\abs{a}_h$, $\widetilde{\abs{b}}_h$, $\abs{c}_h$ & $\mathrm{a}_h$, $\mathrm{b}_h$, $\mathrm{c}_h^\dg$ & \eqref{eq:ah}, \eqref{eq:bh}, \eqref{eq:ch.dg}
      \\
      $\underline{\alpha}$, $\overline{\alpha}$ & $C_{\mathrm{a}}^{-1}2\mu$, $C_{\mathrm{a}}(2\mu+d\lambda^+)$ & \eqref{eq:ah:coercivity.boundedness}
      \\
      $\beta$ & $\beta$ & \eqref{eq:bh:inf-sup}      
      \\
      $\gamma$ & 1 & ---      
      \\
      $(\abs{f}_h^n,\abs{g}_h^n)_{1\le n\le N}$ & $(\overline{\bvec{f}}^n,\overline{g}^n)_{1\le n\le N}$ & \eqref{eq:overline.fhn.ghn}
      \\
      \bottomrule
    \end{tabular}
  \end{small}
\end{table}
\begin{problem}[HHO-DG scheme]\label{prob:hho-dg}
  Let $k\ge 1$.
  The families $\uvec{u}_{h\tau}\coloneq(\uvec{u}_h^n)_{0\le n\le N}\in(\uvec{U}_{h,0}^k)^{N+1}$ and $p_{h\tau}\coloneq(p_h^n)_{0\le n\le N}\in(P_h^k)^{N+1}$ with $P_h^k\coloneq\Poly{k}(\Th)$ are such that, for $n=1,\ldots, N$,
  $$  
    \begin{aligned}
      \mathrm{a}_h(\uvec{u}_h^n,\uvec{v}_h) + \mathrm{b}_h(\uvec{v}_h,p_h^n) 
      &= (\overline{\bvec{f}}^n,\bvec{v}_h)_{L^2(\Omega)^d}
      &\qquad& \forall\uvec{v}_h\in\uvec{U}_{h,0}^k,
      \\ 
      C_0(\delta_t^n p_{h\tau},q_h)_{L^2(\Omega)} - \mathrm{b}_h(\delta_t^n\uvec{u}_{h\tau},q_h) 
      + \mathrm{c}_h^\dg(p_h^n,q_h) &= (\overline{g}^n,q_h)_{L^2(\Omega)}
      &\qquad& \forall q_h\in P_h^k,
    \end{aligned}
  $$
    and, if $C_0=0$, for all $1\le n\le N$,
    \[
    \int_\Omega p_h^n(\bvec{x})\ud\bvec{x}=0.
    \]
    The initial values of the discrete displacement and pore pressure are chosen such that
    \[
    C_0(p_h^0,q_h)_{L^2(\Omega)} - \mathrm{b}_h(\uvec{u}_h^0,q_h) = (\phi^0,q_h)_{L^2(\Omega)}\qquad\forall q_h\in P_h^k.
    \]
\end{problem}
\begin{table}[h]
  \caption{Interpretation of the HHO-DG method within the abstract framework of Section \ref{sec:abstract.framework}
           \label{tab:interpretation_dg}}
  \centering
  \begin{small}
    \renewcommand\arraystretch{1.2}
    \begin{tabular}{ccc}
      \toprule
      Abstract object & Instance & Definition
      \\
      \midrule
      $\bvec{\abs{U}}_h$, $\norm[\bvec{\abs{U}},h]{{\cdot}}$ & $\uvec{U}_h^k$, $\norm[\strain,h]{{\cdot}}$ & \eqref{eq:Uhk0}, \eqref{eq:norm.strain.h}
      \\
      $\abs{P}_h$, $\norm[\abs{P},h]{{\cdot}}$ & $P_h^k=\Poly{k}(\Th)$, $\norm[\dg,h]{{\cdot}}$ & Problem \ref{prob:hho-dg}, \eqref{eq:norm.dg.h}
      \\
      $\abs{r}_h$ & $P_h^k\hookrightarrow L^2(\Omega)$ & ---
      \\
      $\abs{L}_h$ & $\Poly{k}(\Th)$ & ---
      \\
      $\abs{a}_h$, $\abs{b}_h$, $\abs{c}_h$ & $\mathrm{a}_h$, $\mathrm{b}_h$, $\mathrm{c}_h^\hho$ & \eqref{eq:ah}, \eqref{eq:bh}, \eqref{eq:ch.hho}
      \\
      $\underline{\alpha}$, $\overline{\alpha}$ & $C_{\mathrm{a}}^{-1}2\mu$, $C_{\mathrm{a}}(2\mu+d\lambda^+)$ & \eqref{eq:ah:coercivity.boundedness}
      \\
      $\beta$ & $\beta$ & \eqref{eq:bh:inf-sup}
      \\
      $\gamma$ & $\gamma^\dg$ & cf. Section \ref{sec:dg.darcy}
      \\
      $(\abs{f}_h^n,\abs{g}_h^n)_{1\le n\le N}$ & $(\overline{\bvec{f}}^n,\overline{g}^n)_{1\le n\le N}$ & \eqref{eq:overline.fhn.ghn}
      \\
      \bottomrule
    \end{tabular}
  \end{small}
\end{table}


\section{Error analysis}\label{sec:error.analysis}

The family $\hat{\uvec{u}}_{h\tau}=(\hat{\uvec{u}}_h^n)_{0\le n\le N}$ of displacement interpolates is defined setting
\begin{equation}\label{eq:hat.uh}
  \text{
    $\hat{\uvec{u}}_h^0\coloneq\uvec{I}_h^k\bvec{u}(0)$ and
    $\hat{\uvec{u}}_h^n\coloneq I_\tau^n(\uvec{I}_h^k\bvec{u})$ for all $1\le n\le N$.
  }
\end{equation}
For the HHO-HHO method, denoting, for a.e. $t\in[0,\tF]$, by $\underline{\mathfrak{p}}_h^\hho(t)\in\underline{P}_h^k$ the solution of the HHO projection problem \eqref{eq:ch.hho.projection} with $r=p(t)$, the family $\hat{\underline{p}}_{h\tau}^\hho=(\hat{\underline{p}}_h^{\hho,n})_{0\le n\le N}$ of pore pressure interpolates is given by
\begin{equation}\label{eq:hat.ph.hho}
  \text{
    $\hat{\underline{p}}_h^{\hho,0}\coloneq\underline{I}_h^k p(0)$ and
    $\hat{\underline{p}}_h^{\hho,n}\coloneq I_\tau^n\underline{\mathfrak{p}}^\hho_h$ for all $1\le n\le N$.
  }
\end{equation}
We remind the reader that $\hat{p}_{h\tau}^\hho=(\hat{p}_h^{\hho,n})_{0\le n\le N}\in\Poly{k}(\Th)^{N+1}$ is the family of broken polynomial fields obtained from $\hat{\underline{p}}_{h\tau}^\hho$ by patching the element-based pore pressure unknowns according to \eqref{eq:qh}.
Similarly, denoting, for a.e. $t\in[0,\tF]$, by $\mathfrak{p}^\dg_h(t)\in P_h^k$ the solution of the DG projection problem \eqref{eq:ch.dg.projection} with $r=p(t)$ we define the family $\hat{p}_{h\tau}^\dg=(\hat{p}_h^n)_{0\le n\le N}$ of pore pressure interpolates for the HHO-DG method setting
\begin{equation}\label{eq:hat.ph.dg}
  \text{
    $\hat{p}_h^{\dg,0}\coloneq\lproj[h]{k} p(0)$ and
    $\hat{p}_h^{\dg,n}\coloneq I_\tau^n\mathfrak{p}^\dg_h$ for all $1\le n\le N$.
  }
\end{equation}
\begin{theorem}[Error estimates]\label{thm:error.estimate}
  Denote by $(\uvec{u},p)\in L^2(0,\tF;\bvec{U})\times L^2(0,\tF;P)$ the unique solution to \eqref{eq:weak_form}.
  The following error estimates hold:
  \begin{enumerate}
  \item \emph{HHO-HHO scheme.}
    Let $k\ge 0$ and assume $\bvec{u}\in L^2(0,\tF;H^{m+2}(\Th)^d)\cap H^1(0,\tF;H^1(\Omega)^d)$ and $p\in L^2(0,\tF;H^{m+2}(\Th))\cap H^1(0,\tF;L^2(\Omega))$ for some $m\in\{0,\ldots,k\}$.
    Let $(\uvec{u}_{h\tau},\underline{p}_{h\tau})\in(\uvec{U}_{h,0}^k\times\underline{P}_h^k)^{N+1}$ solve 
    Problem \ref{prob:hho-hho} and define the discrete errors
    \[
    \uvec{e}_{h\tau}\coloneq\uvec{u}_{h\tau} - \hat{\uvec{u}}_{h\tau},\qquad 
    \underline{\epsilon}_{h\tau}\coloneq\underline{p}_{h\tau}-\hat{\underline{p}}_{h\tau}^\hho,\qquad
    \text{
      $\underline{z}_h^0\coloneq\underline{0}$ and
      $\underline{z}_h^n\coloneq\sum_{i=1}^n\tau\underline{\epsilon}_h^i$ for all $1\le n\le N$.
    }
    \]
    Then, it holds
    \begin{multline}\label{eq:err.est_hho}
      2\mu\sum_{n=1}^N\tau\norm[\strain,h]{\uvec{e}_h^n}^2
      + \sum_{n=1}^N\tau\left(
      C_0\norm[L^2(\Omega)]{\epsilon_h^n}^2
      +\frac{\mu\beta^2}{(2\mu+d\lambda^+)^2}\norm[L^2(\Omega)]{\epsilon_h^n-\lproj[\Omega]{0}\epsilon_h^n}^2
      \right)
      + \norm[\hho,h]{\underline{z}_h^N}^2
      \\
      \lesssim 
      h^{2(m+1)}\left(
      \mathcal{C}_1\mathfrak{N}_{\bvec{u}} + (\mathcal{C}_1 + \mathcal{C}_2 C_0^2 + 4C_0)\mathfrak{N}_p^\hho
      \right)
      + \tau^2\mathfrak{M};
    \end{multline}
  \item \emph{HHO-DG scheme.}
    Let $k\ge 1$ and assume $\bvec{u}\in L^2(0,\tF;H^{m+2}(\Th)^d)\cap H^1(0,\tF;H^1(\Omega)^d)$ and $p\in L^2(0,\tF;H^{m+1}(\Th))\cap H^1(0,\tF;L^2(\Omega))$ for some $m\in\{0,\ldots,k\}$.
  Further assume elliptic regularity (that is, $\Omega$ convex and $\diff\in\Poly{0}(\Omega)^{d\times d}$),
  let $(\uvec{u}_{h\tau},p_{h\tau})\in(\uvec{U}_{h,0}^k\times P_h^k)^{N+1}$ solve Problem \ref{prob:hho-dg}, and define the discrete errors
  \begin{equation}\label{eq:errors.dg}
    \uvec{e}_{h\tau}\coloneq\uvec{u}_{h\tau} - \hat{\uvec{u}}_{h\tau},\qquad
    \epsilon_{h\tau}\coloneq p_{h\tau}-\hat{p}_{h\tau}^\dg,\qquad
    \text{
      $z_h^0\coloneq 0$ and
      $z_h^n\coloneq\sum_{i=1}^n\tau\epsilon_h^i$ for all $1\le n\le N$.
    }
  \end{equation}
  Then, it holds
  \begin{multline}\label{eq:err.est_dg}
    2\mu\sum_{n=1}^N\tau\norm[\strain,h]{\uvec{e}_h^n}^2
    + \sum_{n=1}^N\tau\left(
    C_0\norm[L^2(\Omega)]{\epsilon_h^n}^2
    + \frac{\mu\beta^2}{(2\mu+d\lambda^+)^2}\norm[L^2(\Omega)]{\epsilon_h^n-\lproj[\Omega]{0}\epsilon_h^n}^2
    \right)
    + \norm[\dg,h]{z_h^N}^2
    \\
    \lesssim
    h^{2(m+1)}\left(
    \mathcal{C}_1\mathfrak{N}_{\bvec{u}} + (\mathcal{C}_1 + \mathcal{C}_2 C_0^2 + 4C_0)\mathfrak{N}_p^\dg
    \right)
    + \tau^2 \mathfrak{M}.
  \end{multline}
  \end{enumerate}
  In \eqref{eq:err.est_hho} and \eqref{eq:err.est_dg}, the hidden constants are independent of $h$, $\tau$, $\mu$, $\lambda$, $C_0$, $\diff$, $\bvec{u}$, and $p$,
  the constants $\mathcal{C}_1, \mathcal{C}_2$ are defined by \eqref{eq:C1.C2.C3} according to 
  Table \ref{tab:interpretation_hho} for the HHO-HHO scheme and according to Table \ref{tab:interpretation_dg} 
  for the HHO-DG scheme,
  and we have introduced the following bounded norms of the continuous solution:
  \begin{align}\label{eq:N.sp.u}
    \mathfrak{N}_{\bvec{u}}
    &\coloneq
    \norm[L^2(0,\tF;H^{m+1}(\Th)^{d\times d})]{\btens{\sigma}(\GRADs\bvec{u})}^2
    + \mu\norm[L^2(0,\tF;H^{m+2}(\Th)^d)]{\bvec{u}}^2,
    \\ \label{eq:N.sp.p}
    \mathfrak{N}_p^\bullet
    &\coloneq\begin{cases}
    \rho^2\sum_{T\in\Th}\rho_T\norm[L^2(0,\tF;H^{m+2}(T))]{p}^2 + \rho^2\norm[L^2(0,\tF;H^{m+1}(\Th))]{p}^2
    & \text{if $\bullet=\hho$},
    \\
    \rho\norm[L^2(0,\tF;H^{m+1}(\Th))]{p}^2 & \text{if $\bullet=\dg$},
    \end{cases}
    \\ \nonumber 
    \mathfrak{M}
    &\coloneq
    (\mathcal{C}_2 C_0^2+4C_0)\norm[H^1(0,\tF;L^2(\Omega))]{p}^2+\mathcal{C}_2\norm[H^1(0,\tF;L^2(\Omega))]{\DIV\bvec{u}}^2.
  \end{align}
\end{theorem}
%
%
\begin{remark}[Time regularity assumption]
The requirements $\bvec{u}\in H^1(0,\tF;H^1(\Omega)^d)$ and $p\in H^1(0,\tF;L^2(\Omega))$ in Theorem \ref{thm:error.estimate} are consistent with the regularity results established by Theorem \ref{thm:reg_est}. Moreover, we remark that the error estimates also hold under the weaker assumption that the maps $t\mapsto p(t)$ and $t\mapsto\DIV\bvec{u}(t)$ are only piecewise $H^1$-regular on the time mesh.
\end{remark}
\begin{proof}[Proof of Theorem \ref{thm:error.estimate}]
  The proof proceeds in several steps: first, we estimate the residual on the mechanical equilibrium equation; then, we split the residual of the flow equation into components that yield purely spatial and temporal errors; finally, we invoke Lemma \ref{lem:abs:err.est} to obtain a basic estimate of the error in terms of these residuals.
  In what follows, the superscript $\bullet\in\{\hho,\dg\}$ is used for method-specific quantities when needed.
  Hidden constants in the inequalities have the same dependencies as in the theorem statement.
  \medskip\\
  \underline{1. \emph{Estimate of the residual on the mechanical equilibrium equation}.}
  For any $1\le n\le N$, the residual $\mathcal{E}_{\bvec{U},h}^{\bullet,n}:\uvec{U}_{h,0}^k\to\Real$ on the mechanical equilibrium equation is defined as follows (cf. \eqref{eq:abs:EUhn}):
  For all $\uvec{v}_h\in\uvec{U}_{h,0}^k$,
  \[
  \mathcal{E}_{\bvec{U},h}^{\bullet,n}(\uvec{v}_h)
  \coloneq (\overline{\bvec{f}}^n,\bvec{v}_h)_{L^2(\Omega)^d}
  - \mathrm{a}_h(\hat{\uvec{u}}_h^n,\uvec{v}_h)
  - \mathrm{b}_h(\uvec{v}_h,\hat{p}_h^{\bullet,n}).
  \]
  We next prove that
  \begin{equation}\label{eq:est.mechanics.residual}
    \boxed{
      \sum_{n=1}^N\tau\norm[\strain,h,*]{\mathcal{E}_{\bvec{U},h}^{\bullet,n}}^2\lesssim
      h^{2(m+1)}\left(
      \mathfrak{N}_{\bvec{u}}
      + \mathfrak{N}_p^\bullet
      \right).
    }
  \end{equation}
  We have, for all $\uvec{v}_h\in\uvec{U}_{h,0}^k$ such that $\norm[\strain,h]{\uvec{v}_h}=1$ and all $1\le n\le N$,
  \[
  \begin{aligned}
    \mathcal{E}_{\bvec{U},h}^{\bullet,n}(\uvec{v}_h)
    &=\frac1\tau\int_{t^{n-1}}^{t^n}\left(
    (\bvec{f}(t),\bvec{v}_h)_{L^2(\Omega)^d}
    - \mathrm{a}_h(\uvec{I}_h^k\bvec{u}(t),\uvec{v}_h)
    - \mathrm{b}_h(\uvec{v}_h,\mathfrak{p}_h^{\bullet}(t))
    \right)\ud t
    \\
    &=\frac1\tau\int_{t^{n-1}}^{t^n}\left(
    - (\DIV\btens{\sigma}(\GRADs\bvec{u}(t)),\bvec{v}_h)_{L^2(\Omega)^d}
    - \mathrm{a}_h(\uvec{I}_h^k\bvec{u}(t),\uvec{v}_h)
    + (\GRAD p(t),\bvec{v}_h)_{L^2(\Omega)^d}    
    - \mathrm{b}_h(\uvec{v}_h,\mathfrak{p}_h^{\bullet}(t))
    \right)\ud t
    \\
    &=\frac1\tau\int_{t^{n-1}}^{t^n}\left(
    \mathcal{E}_{\mathrm{a},h}(\bvec{u}(t);\uvec{v}_h)
    + \mathcal{E}_{\mathrm{b},h}^{\bullet}(p(t);\uvec{v}_h)
    \right)\ud t,
  \end{aligned}
  \]
  where we have used the definition \eqref{eq:I.tau.n} of the time interpolator in the first equality,
  the fact that \eqref{eq:biot:strong:mechanics} holds for a.e. $t\in(0,\tF)$ and $\bvec{x}\in\Omega$ to substitute $\bvec{f}$ with the left-hand side of this equation (recall that we have assumed $C_{\rm bw}=1$),
  and the definitions \eqref{eq:ah:consistency.error} of $\mathcal{E}_{\mathrm{a},h}$ and \eqref{eq:bh.hho:consistency.error} (if $\bullet=\hho$) or \eqref{eq:bh.dg:consistency.error} (if $\bullet=\dg$) of $\mathcal{E}_{\mathrm{b},h}^\bullet$ to conclude.
  We continue taking the absolute value, using the consistency estimate \eqref{eq:ah:consistency}, the definition of the dual norm $\norm[\strain,h,*]{{\cdot}}$, and \eqref{eq:bh:consistency:hho} (if $\bullet=\hho)$ or \eqref{eq:bh:consistency:dg} (if $\bullet=\dg$) to obtain
  \begin{equation}\label{eq:est.Euh.t}
    \norm[\strain,h,*]{\mathcal{E}_{\bvec{U},h}^{\bullet,n}}
    \lesssim\frac{h^{m+1}}{\tau}\int_{t^{n-1}}^{t^n} \mathcal{N}_{\bvec{u},p}^\bullet(t)\ud t
  \end{equation}
  with, for a.e. $t\in(0,\tF)$,
  \[
  \mathcal{N}_{\bvec{u},p}^\hho(t)^2\coloneq
  \seminorm[H^{m+1}(\Th)^{d\times d}]{\btens{\sigma}(\GRADs\bvec{u}(t))}^2
  + 2\mu\seminorm[H^{m+2}(\Th)^d]{\bvec{u}(t)}^2
  + \rho^2\left(
  \sum_{T\in\Th}\rho_T\seminorm[H^{m+2}(T)]{p(t)}^2
  + \seminorm[H^{m+1}(\Th)]{p(t)}^2
  \right)
  \]
  and
  \[
  \mathcal{N}_{\bvec{u},p}^\dg(t)^2\coloneq
  \seminorm[H^{m+1}(\Th)^{d\times d}]{\btens{\sigma}(\GRADs\bvec{u}(t))}^2
  + 2\mu\seminorm[H^{m+2}(\Th)^d]{\bvec{u}(t)}^2
  + \rho\seminorm[H^{m+1}(\Th)]{p(t)}^2.
  \]    
  Squaring \eqref{eq:est.Euh.t}, multiplying by $\tau$, summing over $1\le n\le N$, and using a Cauchy--Schwarz inequality, we obtain
  \[
  \sum_{n=1}^N\tau\norm[\strain,h,*]{\mathcal{E}_{\bvec{U},h}^{\bullet,n}}^2
  \lesssim h^{2(m+1)}\sum_{n=1}^N\frac1\tau\left|
  \int_{t^{n-1}}^{t^n}\mathcal{N}_{\bvec{u},p}^\bullet(t)\ud t
  \right|^2
  \le h^{2(m+1)}\int_{0}^{\tF}|\mathcal{N}_{\bvec{u},p}^\bullet(t)|^2\ud t,
  \]
  from which \eqref{eq:est.mechanics.residual} follows recalling the definitions \eqref{eq:N.sp.u} of $\mathfrak{N}_{\bvec{u}}$ and \eqref{eq:N.sp.p} of $\mathfrak{N}_p^\bullet$.
  \medskip\\
  \underline{2. \emph{Estimate of the residual on the flow equation.}}
  For any $1\le n\le N$, we split the residual \eqref{eq:abs:EPhtaun} of the flow equation as follows:
  \[
  \mathcal{E}_{P,h\tau}^{\bullet,n} = \mathcal{E}_{P,h}^{\bullet,n} + \mathcal{E}_{P,\tau}^{\bullet,n},
  \]
  where, for all $q_h\in\Poly{k}(\Th)$,
  \begin{align}\label{eq:EPhn}
    (\mathcal{E}_{P,h}^{\bullet,n},q_h)_{L^2(\Omega)}
    &\coloneq
    C_0(I_\tau^n p,q_h)_{L^2(\Omega)}
    - C_0(\hat{p}_h^{\bullet,0},q_h)_{L^2(\Omega)}
    -\sum_{i=1}^n\tau C_0(\delta_t^i \hat{p}_{h\tau}^\bullet,q_h)_{L^2(\Omega)},
    \\ \label{eq:EPtaun}
    (\mathcal{E}_{P,\tau}^{\bullet,n},q_h)_{L^2(\Omega)}
    &\coloneq
    \sum_{i=1}^n\tau\Big[
      (\overline{g}^i +I_\tau^i(\DIV(\diff\GRAD p)),q_h)_{L^2(\Omega)}
      + \mathrm{b}_h(\delta_t^i \hat{\uvec{u}}_{h\tau},q_h)
      \Big]
    - C_0(I_\tau^n p - \hat{p}_h^{\bullet,0},q_h)_{L^2(\Omega)}.
  \end{align}
  We next proceed to estimate each component of the residual separately.
  \medskip\\
  \underline{2a. \emph{Estimate of $\mathcal{E}_{P,h}^{\bullet,n}$.}}
  We start by proving that
  \begin{equation}\label{eq:est.flow.residual.h}
    \boxed{
      \sum_{n=1}^N\tau\norm[L^2(\Omega)]{\mathcal{E}_{P,h}^{\bullet,n}}^2
      \lesssim C_0^2 h^{2(m+1)}\mathfrak{N}_p^\bullet,
    }
  \end{equation}
  showing that the error associated with $\mathcal{E}_{P,h}^{\bullet,n}$ is purely spatial.
  A discrete integration by parts in time in \eqref{eq:EPhn} gives, for all $q_h\in\Poly{k}(\Th)$,
  \[
  (\mathcal{E}_{P,h}^{\bullet,n},q_h)_{L^2(\Omega)}
  =C_0(I_\tau^n p,q_h)_{L^2(\Omega)}-C_0(\hat{p}_h^{\bullet,n},q_h)_{L^2(\Omega)} 
  =C_0(I_\tau^n (p-\mathfrak{p}_h^\bullet),q_h)_{L^2(\Omega)}.
  \]  
  Making $q_h=\mathcal{E}_{P,h}^{\bullet,n}$, using a Cauchy--Schwarz inequality in the right-hand side, simplifying, and squaring the resulting inequality, we obtain
  \begin{equation}\label{eq:est.normL2.EPh:basic}
    \begin{aligned}
      \norm[L^2(\Omega)]{\mathcal{E}_{P,h}^{\bullet,n}}^2
      &\le C_0^2\norm[L^2(\Omega)]{I_\tau^n (p-\mathfrak{p}_h^\bullet)}^2
      \\
      &= C_0^2\int_\Omega\left(
      \frac1\tau\int_{t^{n-1}}^{t^n} \left( p(\bvec{x},t)-\mathfrak{p}_h^\bullet(\bvec{x},t) \right) \ud t
      \right)^2\ud\bvec{x}
      \\
      &\le C_0^2\int_\Omega \frac1{\tau^2} \left(\int_{t^{n-1}}^{t^n} \ud t\right)
      \left( \int_{t^{n-1}}^{t^n} |p(\bvec{x},t)-\mathfrak{p}_h^\bullet(\bvec{x},t)|^2 \ud t \right) \ud \bvec{x}
      \\
      &= \frac{C_0^2}\tau\int_\Omega\int_{t^{n-1}}^{t^n} |p(\bvec{x},t)-\mathfrak{p}_h^\bullet(\bvec{x},t)|^2 \ud t \ud \bvec{x}
      = \frac{C_0^2}{\tau}\norm[L^2(t^{n-1},t^n;L^2(\Omega))]{p-\mathfrak{p}_h^\bullet}^2,
    \end{aligned}
  \end{equation}
  where we have used the definition \eqref{eq:I.tau.n} of $I_\tau^n$ in the second line and a Cauchy--Schwarz inequality in the third line.
  The rest of the proof proceeds using estimates that are method-specific.
  For the HHO-HHO scheme, using \eqref{eq:ch.hho.projection:err.est} with $r=p(t)$, we have
  \[
  \norm[L^2(\Omega)]{\mathcal{E}_{P,h}^{\hho,n}}^2
  \lesssim \frac{C_0^2\rho^2}{\tau}\int_{t^{n-1}}^{t^n}\left(
  \sum_{T\in\Th}\rho_Th_T^{2(m+1)}\seminorm[H^{m+2}(T)]{p(t)}^2 + h^{2(m+1)}\seminorm[H^{m+1}(\Th)]{p(t)}^2
  \right)\ud t.
  \]
  Multiplying this inequality by $\tau$ and summing over $n=1,\ldots,N$, we arrive at
  \[
  \begin{aligned}
    \sum_{n=1}^N\tau\norm[L^2(\Omega)]{\mathcal{E}_{P,h}^{\hho,n}}^2
    &\lesssim C_0^2\rho^2\sum_{n=1}^N\int_{t^{n-1}}^{t^n}\left(
    \sum_{T\in\Th}\rho_Th_T^{2(m+1)}\seminorm[H^{m+2}(T)]{p(t)}^2 + h^{2(m+1)}\seminorm[H^{m+1}(\Th)]{p(t)}^2
    \right)\ud t
    \\
    &\le C_0^2 h^{2(m+1)}\mathfrak{N}_p^\hho.
  \end{aligned}
  \]
  For the HHO-DG scheme, recalling \eqref{eq:ch.dg.projection:err.est}, we can continue from \eqref{eq:est.normL2.EPh:basic} writing
  \[
  \norm[L^2(\Omega)]{\mathcal{E}_{P,h}^{\dg,n}}^2
  \lesssim C_0^2 \rho \frac{h^{2(m+1)}}{\tau}\int_{t^{n-1}}^{t^n}\seminorm[H^{m+1}(\Th)]{p(t)}^2 \ud t.
  \]
  Squaring this inequality, multiplying by the time step $\tau$, and summing over $n=1,\ldots,N$, we obtain
  \[
  \sum_{n=1}^N\tau\norm[L^2(\Omega)]{\mathcal{E}_{P,h}^{\dg,n}}^2
  \lesssim C_0^2 \rho h^{2(m+1)}\sum_{n=1}^N\left(
  \int_{t^{n-1}}^{t^n}\seminorm[H^{m+1}(\Th)]{p(t)}^2\ud t
  \right)
  \le C_0^2 h^{2(m+1)}\mathfrak{N}_p^\dg.
  \]
  This concludes the proof of \eqref{eq:est.flow.residual.h}.
  \medskip\\
  \underline{2b. \emph{Estimate of $\mathcal{E}_{P,\tau}^{\bullet,n}$.}}
  We proceed by proving that
  \begin{equation}\label{eq:est.flow.residual.tau}
    \boxed{
      \sum_{n=1}^N\tau\norm[L^2(\Omega)]{\mathcal{E}_{P,\tau}^{\bullet,n}}^2
      \lesssim \tau^2\left(
      C_0\norm[H^1(0,\tF;L^2(\Omega))]{p} + \norm[H^1(0,\tF;L^2(\Omega))]{\DIV\bvec{u}}
      \right)^2,
    }
  \end{equation}
  showing that the error associated with $\mathcal{E}_{P,\tau}^{\bullet,n}$ is purely temporal.
First, we remark that, owing to definitions \eqref{eq:hat.uh} of $\hat{\uvec{u}}_h^0$, \eqref{eq:hat.ph.hho} (for the HHO-HHO scheme) or \eqref{eq:hat.ph.dg} (for the HHO-DG scheme) of $\hat{p}_h^{\bullet,0}$, using the consistency property \eqref{eq:bh:consistency:1} of 
  $\mathrm{b}_h$, and recalling the definition \eqref{eq:a.b.c} of $b$ and the initial condition \eqref{eq:weak_form.initial}, it holds, for all $q_h\in\Poly{k}(\Th)$,
  \begin{equation}\label{eq:est.Eptau.initial}
  C_0(\hat{p}_h^{\bullet,0},q_h)_{L^2(\Omega)}-\mathrm{b}_h(\hat{\uvec{u}}_h^0,q_h) =
  (C_0 p(0) + \DIV \bvec{u}(0), q_h)_{L^2(\Omega)}
  =(\phi^0,q_h)_{L^2(\Omega)}.
  \end{equation}
  Invoking the definitions \eqref{eq:EPtaun} of $\mathcal{E}_{P,\tau}^{\bullet,n}$, \eqref{eq:overline.fhn.ghn} of $\overline{g}^n$, and \eqref{eq:I.tau.n} of the time interpolator, telescoping out the appropriate quantities, and using \eqref{eq:est.Eptau.initial}, it is inferred that
  \begin{equation}\label{eq:est.EPtau:basic}
    \begin{aligned}
      (\mathcal{E}_{P,\tau}^{\bullet,n},q_h)_{L^2(\Omega)}
      &=\int_0^{t^n}\left(
        (g(t) + \DIV(\diff\GRAD p(t)),q_h)_{L^2(\Omega)}
        \right)\ud t
        + \mathrm{b}_h(\hat{\uvec{u}}_h^n,q_h) 
        - C_0(I_\tau^n p,q_h)_{L^2(\Omega)}  
        + (\phi^0, q_h)_{L^2(\Omega)}
      \\
      &= \int_0^{t^n}
      \big(\dt (C_0 p(t) + \DIV \bvec{u}(t)),q_h\big)_{L^2(\Omega)}\ud t
      -\big(I_\tau^n (C_0 p + \DIV\bvec{u}),q_h \big)_{L^2(\Omega)}
      + (\phi^0, q_h)_{L^2(\Omega)}
      \\
      &= \left( C_0 p(t^n) + \DIV \bvec{u}(t^n), q_h\right)_{L^2(\Omega)}
      -\left(I_\tau^n (C_0 p + \DIV\bvec{u}),q_h \right)_{L^2(\Omega)},
    \end{aligned}
  \end{equation}
  where we have substituted $g(t)$ and $\hat{\uvec{u}}_h^n$ according to \eqref{eq:biot:strong:flow} 
  and \eqref{eq:hat.uh}, respectively, and used \eqref{eq:bh:consistency:1} to pass to the second line, then we have 
  applied the fundamental theorem of calculus to the integral and used \eqref{eq:weak_form.initial} to conclude.
  Hence, making $q_h=\mathcal{E}_{P,\tau}^{\bullet,n}$ in \eqref{eq:est.EPtau:basic}, using a Cauchy--Schwarz inequality 
  in the right-hand side, simplifying, and squaring, we infer
  \begin{equation}\label{eq:EPtau.basic}
    \norm[L^2(\Omega)]{\mathcal{E}_{P,\tau}^{\bullet,n}}^2
    \le \norm[L^2(\Omega)]{C_0 p(t^n) + \DIV \bvec{u}(t^n) - I_\tau^n (C_0 p + \DIV\bvec{u})}^2
    =\norm[L^2(\Omega)]{\phi(t^n) - I_\tau^n \phi}^2,
  \end{equation}
  with $\phi\coloneq C_0 p + \DIV\bvec{u}$. We next observe that
  \[
  \begin{aligned}
    \phi(t^n) - I_\tau^n \phi
    &= \phi(t^n) - \frac1\tau\int_{t^{n-1}}^{t^n}\phi(t)\ud t
    \\
    &=
    \phi(t^n)
    - \frac{1}{\tau}\int_{t^{n-1}}^{t^n}\left(\phi(t^n) - \int_t^{t^n}\hspace{-1ex}\dt\phi(s)\ud s\right) \ud t
    \\
    &= \frac{1}{\tau}\int_{t^{n-1}}^{t^n}\int_t^{t^n}\hspace{-1ex}\dt\phi(s)\ud s\ud t
    \le\int_{t^{n-1}}^{t^n}|\dt\phi(t)|\ud t.
  \end{aligned}
  \]
  Combining this result with the Jensen inequality, we infer
  \begin{equation}\label{eq:EPtau.jensen}
    \norm[L^2(\Omega)]{\phi(t^n) - I_\tau^n \phi}^2
    \le\int_\Omega\left(
    \int_{t^{n-1}}^{t^n}|\dt\phi(\bvec{x},t)| \ud t
    \right)^2\ud\bvec{x}
    \le\tau\int_{t^{n-1}}^{t^n}\norm[L^2(\Omega)]{\dt\phi(t)}^2\ud t
    \le\tau\norm[H^1(t^{n-1},t^n;L^2(\Omega))]{\phi}^2.
  \end{equation}
  Plugging \eqref{eq:EPtau.jensen} into \eqref{eq:EPtau.basic}, multiplying by $\tau$, and summing over $1\le n\le N$, we can write
  \[
    \sum_{n=1}^N\tau\norm[L^2(\Omega)]{\mathcal{E}_{P,\tau}^{\bullet,n}}^2
    \le\tau^2\sum_{n=1}^N\norm[H^1(t^{n-1},t^n;L^2(\Omega))]{\phi}^2
    = \tau^2\norm[H^1(0,\tF;L^2(\Omega))]{\phi}^2,
    \]
    which yields \eqref{eq:est.flow.residual.tau} after observing that
    $\norm[H^1(0,\tF;L^2(\Omega))]{\phi}\le C_0\norm[H^1(0,\tF;L^2(\Omega))]{p} + \norm[H^1(0,\tF;L^2(\Omega))]{\DIV\bvec{u}}$.
    \medskip\\
  \underline{3. \emph{Proof of the error estimates}.}
  Since the discretisations introduced in Section \ref{sec:hho} satisfy Assumptions \ref{ass:discrete.setting} and \ref{ass:stability} and the interpolates 
  defined by \eqref{eq:hat.uh}, \eqref{eq:hat.ph.hho} (for the HHO-HHO scheme), and \eqref{eq:hat.ph.dg} (for the HHO-DG scheme) satisfy  Assumption \ref{ass:interpolate}, estimate \eqref{eq:abs:err.est} holds with $\norm[\bvec{\abs{U}},h,*]{\mathcal{E}_{\bvec{\abs{U}},h}^n}$ and $\norm[L^2(\Omega)]{\mathcal{E}_{\abs{P},h\tau}^n}$ replaced by $\norm[\strain,h,*]{\mathcal{E}_{\bvec{U},h}^{\bullet,n}}$ and $\norm[L^2(\Omega)]{\mathcal{E}_{P,h\tau}^{\bullet,n}}$, respectively. 
  Moreover, owing to \eqref{eq:est.Eptau.initial}, we have for both schemes
  \begin{equation}\label{eq:est.E0h}
    (\mathcal{E}_{0,h}^\bullet,q_h)_{L^2(\Omega)}\coloneq
    (\phi^0,q_h)_{L^2(\Omega)}- C_0(\hat{p}_h^{\bullet,0},q_h)_{L^2(\Omega)} + \mathrm{b}_h(\hat{\uvec{u}}_h^0,q_h)
    = 0\qquad\forall q_h\in\Poly{k}(\Th).
  \end{equation}
  Hence, it only remains to bound the fourth term in the right-hand side of \eqref{eq:abs:err.est}, namely
  $$
  \term\coloneq 
  \mathcal{C}_3 \sum_{n=1}^N \tau \norm[L^2(\Omega)]{\lproj[\Omega]{0}\mathcal{E}_{P,h\tau}^{\bullet,n}}^2
  =\frac4{C_0} \sum_{n=1}^N \tau 
  \norm[L^2(\Omega)]{
    \lproj[\Omega]{0}\mathcal{E}_{P,h}^{\bullet,n}
    + \lproj[\Omega]{0}\mathcal{E}_{P,\tau}^{\bullet,n}
  }^2,
  $$
  in such a way that $\term=0$ when $C_0=0$.
  Proceeding as in \eqref{eq:est.EPtau:basic} and using $\lproj[\Omega]{0}(\DIV\bvec{u}(t)) = 0$ for all $t\in(0,\tF]$ (combine the Stokes theorem with the clamped boundary condition \eqref{eq:biot:strong:bc.u}), it is inferred that, for all $1\le n \le N$ and $q_h\in\Poly{k}(\Th)$,
  $$
  (\lproj[\Omega]{0}\mathcal{E}_{P,\tau}^{\bullet,n},q_h)_{L^2(\Omega)} =
  (\mathcal{E}_{P,\tau}^{\bullet,n},\lproj[\Omega]{0}q_h)_{L^2(\Omega)} =
  C_0 ( p(t^n) - I_\tau^n p, \lproj[\Omega]{0} q_h)_{L^2(\Omega)}.
  $$
  Therefore, following the arguments used in Point 3. above, we obtain
  \begin{equation}\label{eq:est.Ep.tau_av}
    \sum_{n=1}^N\tau\norm[L^2(\Omega)]{\lproj[\Omega]{0}\mathcal{E}_{P,\tau}^n}^2
    \le \tau^2 C_0^2\norm[H^1(0,\tF;L^2(\Omega))]{p}^2.
  \end{equation}
  Using the triangle inequality, the boundedness of $\lproj[\Omega]{0}$, and estimates \eqref{eq:est.flow.residual.h} and
  \eqref{eq:est.Ep.tau_av}, leads to
  \begin{equation}\label{eq:est.Ep_av}
    \term\lesssim
    C_0^{-1}\left(
    \sum_{n=1}^N\tau\norm[L^2(\Omega)]{\lproj[\Omega]{0}\mathcal{E}_{P,h}^{\bullet,n}}^2
    +\sum_{n=1}^N\tau\norm[L^2(\Omega)]{\lproj[\Omega]{0}\mathcal{E}_{P,\tau}^{\bullet,n}}^2
    \right)
    \le 
    C_0 \left( h^{2(m+1)}\mathfrak{N}_p^\bullet
    +\tau^2\norm[H^1(0,\tF;L^2(\Omega))]{p}^2 \right). 
  \end{equation}
  It follows from the previous bound that $\term=0$ in the case $C_0=0$.
  Plugging \eqref{eq:est.mechanics.residual}, \eqref{eq:est.flow.residual.h}, \eqref{eq:est.flow.residual.tau},  \eqref{eq:est.E0h}, and \eqref{eq:est.Ep_av} into \eqref{eq:abs:err.est} yields, respectively, \eqref{eq:err.est_hho} for the HHO-HHO scheme and \eqref{eq:err.est_dg} for the HHO-DG scheme.
\end{proof}


\section{Numerical examples}\label{sec:numerical.examples}

In order to confirm the convergence rates predicted in Theorem \ref{thm:error.estimate}, we rely on a manufactured regular solution of a two-dimensional incompressible Biot problem ($C_0=0$) with physical parameters $\mu=1$, $\lambda=1$, and $\diff = \kappa \Id$ with $\kappa \in\{ 1, 10^{-6}\}$.
The exact displacement $\bvec{u}$ and exact pressure $p$ are given by, respectively,
\[
\bvec{u}(\bvec{x},t)
= \begin{pmatrix}
  -\sin(\pi t)\cos(\pi x_1)\cos(\pi x_2) \\ \sin(\pi t)\sin(\pi x_1)\sin(\pi x_2)
\end{pmatrix},
\qquad
p(\bvec{x},t)
= - \cos(\pi t)\sin(\pi x_1)\cos(\pi x_2).
\]
The corresponding volumetric load and source terms are
\[
\bvec{f}(\bvec{x},t)
= (6 \pi^2\sin(\pi t)+ \pi \cos(\pi t))\begin{pmatrix}
  -\cos(\pi x_1)\cos(\pi x_2)\\ \sin(\pi x_1)\sin(\pi x_2)
\end{pmatrix},\quad
g(\bvec{x},t) =  2 (1-\kappa) \pi^2 \cos(\pi t) \sin(\pi x_1) \cos(\pi x_2). 
\]
We consider HHO-HHO and HHO-DG dicretisations of degree $k\in\{1,2,3\}$ over a trapezoidal elements mesh sequence of the unit square $\Omega=(0,1)^{2}$.
The time discretisation is based on Backward Differentiation Formulas (BDF) of order $(k+1)$
and time integration is carried out over the interval $\lbrack 0,\tF=1)$ with a fixed time step $\tau=10^{-3}$.  
The domain boundary $\partial\Omega$ is split into two halves where Dirichlet--Neumann and Neumann--Dirichlet boundary conditions for the displacement-pressure couple are imposed, respectively, according to the exact solution. 
Initial conditions are specified by means of $L^2$-projections over mesh elements and mesh faces. Initialization is performed at several time points ($t_i=-\tau \, i, \; i=1,...,k+1$), in agreement with the BDF order.   

The convergence rates for $\kappa=1$ and $\kappa=10^{-6}$ are reported in Tables \ref{tab:unitdiff} and \ref{tab:smalldiff}, respectively. 
Various norms of the errors are reported in Tables \ref{tab:unitdiff} and \ref{tab:smalldiff} with the convention that, for a family $x_{h\tau}=(x_h^n)_{1\le n\le N}\in X_h^N$ with $X_h$ generic vector space, we let
$$
\norm[2,\star]{x_{h\tau}}\coloneq \left(\sum_{n=1}^N \tau\norm[\star]{x_h^n}^2 \right)^{\frac12},
$$
where $\norm[\star]{{\cdot}}$ denotes a norm on $X_h$.
Each error measure is accompanied by the corresponding estimated order of convergence (EOC).
For the sake of completeness, in the first two columns of the tables, we report the size of the linear system 
(denoted by ${\rm \#dofs}$) and the number of non-zero entries in the matrix (denoted by ${\rm \#nz}$).
The numerical results confirm asymptotic convergence rates that are in agreement with theoretical predictions. 
It is interesting to remark that HHO-HHO discretisations provide convergence rates of $(k+2)$ and $(k+1)$ for the pressure error when $\kappa{=}1$ and $\kappa{=}10^{-6}$, respectively, whereas HHO-DG formulations does not surpass $(k+1)$ in case of unit permeabilities. 
Similarly, the $L^2(0,\tF; L^2(\Omega)^d)$-norm of the displacement error shows asymptotic convergence rates of $(k+2)$ for all but the unit permeability $k=1$ HHO-DG formulation, where the convergence rate tops at $(k+1)$.
This behaviour will be investigated in future works.

\begin{table}
  \centering
  \begin{small}
    \begin{tabular}{lccccccccc}
      \toprule
      discr &  k & ${\rm \#dofs}$      & ${\rm \#nz}$      & $\norm[2,\strain,h]{\uvec{e}_{h\tau}}$    & EOC    & 
      $\norm[2,L^2(\Omega)^d]{\bvec{e}_{h\tau}}$    & EOC    & $\norm[2,L^2(\Omega)]{\epsilon_{h\tau}}$  & EOC   \\

      \midrule
      \multirow{5}{*}{\textbf{HHO-DG}} & 
      \multirow{5}{*}{\textbf{1}} & 
      208        & 5824       & 6.69e-03   & --         & 6.82e-04   & --         & 1.13e-03   & --          \\ 
      & & 768        & 23328      & 1.71e-03   & 1.97       & 7.39e-05   & 3.21       & 3.23e-04   & 1.80        \\ 
      & & 2944       & 93376      & 4.03e-04   & 2.08       & 9.97e-06   & 2.89       & 7.49e-05   & 2.11        \\ 
      & & 11520      & 373632     & 9.82e-05   & 2.04       & 1.81e-06   & 2.46       & 1.90e-05   & 1.98        \\ 
      & & 45568      & 1494784    & 2.45e-05   & 2.01       & 3.95e-07   & 2.19       & 4.68e-06   & 2.02        \\ 

      \midrule
      \multirow{5}{*}{\textbf{HHO-HHO}} &
      \multirow{5}{*}{\textbf{1}} &

      240        & 8352       & 6.69e-03   & --         & 6.81e-04   & --         & 4.50e-04   & --          \\ 
      & & 864        & 32832      & 1.70e-03   & 1.97       & 7.06e-05   & 3.27       & 5.68e-05   & 2.99        \\ 
      & & 3264       & 130176     & 4.03e-04   & 2.08       & 8.22e-06   & 3.10       & 6.66e-06   & 3.09        \\ 
      & & 12672      & 518400     & 9.80e-05   & 2.04       & 1.02e-06   & 3.01       & 8.04e-07   & 3.05        \\ 
      & & 49920      & 2068992    & 2.44e-05   & 2.01       & 1.29e-07   & 2.98       & 6.54e-08   & 3.62        \\ 

      \midrule
      \multirow{5}{*}{\textbf{HHO-DG}} & 
      \multirow{5}{*}{\textbf{2}} & 

      336        & 15264      & 8.98e-04   & --         & 4.90e-05   & --         & 1.20e-04   & --          \\ 
      & & 1248       & 61632      & 1.11e-04   & 3.02       & 2.59e-06   & 4.24       & 1.40e-05   & 3.11        \\ 
      & & 4800       & 247680     & 1.44e-05   & 2.94       & 1.75e-07   & 3.89       & 1.82e-06   & 2.94        \\ 
      & & 18816      & 993024     & 1.81e-06   & 2.99       & 1.11e-08   & 3.98       & 2.31e-07   & 2.98        \\ 
      & & 74496      & 3976704    & 2.25e-07   & 3.01       & 7.55e-10   & 3.87       & 2.88e-08   & 3.00        \\ 

      \midrule
      \multirow{5}{*}{\textbf{HHO-HHO}} &
      \multirow{5}{*}{\textbf{2}} &

      360        & 18792      & 8.98e-04   & --         & 4.89e-05   & --         & 2.71e-05   & --          \\ 
      & & 1296       & 73872      & 1.11e-04   & 3.02       & 2.59e-06   & 4.24       & 1.14e-06   & 4.57        \\ 
      & & 4896       & 292896     & 1.44e-05   & 2.94       & 1.74e-07   & 3.89       & 6.37e-08   & 4.16        \\ 
      & & 19008      & 1166400    & 1.81e-06   & 2.99       & 1.10e-08   & 3.98       & 4.15e-09   & 3.94        \\ 
      & & 74880      & 4655232    & 2.25e-07   & 3.01       & 7.52e-10   & 3.87       & 2.95e-10   & 3.82        \\ 

      \midrule
      \multirow{5}{*}{\textbf{HHO-DG}} & 
      \multirow{5}{*}{\textbf{3}} & 

      480        & 31488      & 9.69e-05   & --         & 6.00e-06   & --         & 1.09e-05   & --          \\ 
      & & 1792       & 128128     & 6.63e-06   & 3.87       & 1.27e-07   & 5.57       & 8.30e-07   & 3.72        \\ 
      & & 6912       & 516864     & 3.77e-07   & 4.14       & 3.48e-09   & 5.19       & 4.86e-08   & 4.09        \\ 
      & & 27136      & 2076160    & 2.42e-08   & 3.96       & 1.12e-10   & 4.96       & 3.15e-09   & 3.95        \\ 
      & & 107520     & 8322048    & 1.49e-09   & 4.02       & 3.44e-12   & 5.02       & 1.96e-10   & 4.01        \\ 

      \midrule
      \multirow{5}{*}{\textbf{HHO-HHO}} &
      \multirow{5}{*}{\textbf{3}} &

      480        & 33408      & 9.69e-05   & --         & 5.99e-06   & --         & 2.08e-06   & --          \\ 
      & & 1728       & 131328     & 6.63e-06   & 3.87       & 1.26e-07   & 5.57       & 6.60e-08   & 4.98        \\ 
      & & 6528       & 520704     & 3.77e-07   & 4.14       & 3.48e-09   & 5.19       & 2.41e-09   & 4.78        \\ 
      & & 25344      & 2073600    & 2.42e-08   & 3.96       & 1.12e-10   & 4.96       & 8.15e-11   & 4.88        \\ 
      & & 99840      & 8275968    & 1.49e-09   & 4.02       & 3.44e-12   & 5.02       & 2.43e-12   & 5.07        \\ 
      \bottomrule
    \end{tabular}
  \end{small}
  \caption{\label{tab:unitdiff} Convergence rates for HHO-HHO and HHO-DG discretisations based on
    manufactured solutions of the Biot problem, see text for details. Relevant parameters are $\mu = \lambda = 1$ and  
    $\diff = \kappa \Id$, with $\kappa = 1$.}
\end{table}

\begin{table}
  \centering
  \begin{small}
    \begin{tabular}{lccccccccc}
      \toprule
      discr &  k & ${\rm \#dofs}$     & ${\rm \#nz}$      & $\norm[2,\strain,h]{\uvec{e}_{h\tau}}$    & EOC    & 
      $\norm[2,L^2(\Omega)^d]{\bvec{e}_{h\tau}}$    & EOC    & $\norm[2,L^2(\Omega)]{\epsilon_{h\tau}}$  & EOC   \\

      \midrule
      \multirow{5}{*}{\textbf{HHO-DG}} & 
      \multirow{5}{*}{\textbf{1}} & 

      208        & 5824       & 6.83e-03   & --         & 5.11e-04   & --         & 7.52e-03   & --           \\ 
      & & 768        & 23328      & 1.70e-03   & 2.00       & 5.05e-05   & 3.34       & 1.49e-03   & 2.33         \\ 
      & & 2944       & 93376      & 4.02e-04   & 2.08       & 6.00e-06   & 3.07       & 2.75e-04   & 2.44         \\ 
      & & 11520      & 373632     & 9.74e-05   & 2.05       & 7.02e-07   & 3.10       & 5.70e-05   & 2.27         \\ 
      & & 45568      & 1494784    & 2.42e-05   & 2.01       & 1.07e-07   & 2.72       & 1.24e-05   & 2.19         \\ 

      \midrule
      \multirow{5}{*}{\textbf{HHO-HHO}} &
      \multirow{5}{*}{\textbf{1}} &

      240        & 8352       & 6.83e-03   & --         & 5.11e-04   & --         & 5.25e-03   & --           \\ 
      & & 864        & 32832      & 1.70e-03   & 2.00       & 5.05e-05   & 3.34       & 9.26e-04   & 2.50         \\ 
      & & 3264       & 130176     & 4.02e-04   & 2.08       & 6.01e-06   & 3.07       & 1.70e-04   & 2.44         \\ 
      & & 12672      & 518400     & 9.74e-05   & 2.05       & 7.02e-07   & 3.10       & 3.69e-05   & 2.21         \\ 
      & & 49920      & 2068992    & 2.42e-05   & 2.01       & 1.07e-07   & 2.71       & 8.38e-06   & 2.14         \\ 

      \midrule
      \multirow{5}{*}{\textbf{HHO-DG}} &
      \multirow{5}{*}{\textbf{2}} &

      336        & 15264      & 9.60e-04   & --         & 4.76e-05   & --         & 1.24e-03   & --           \\ 
      &  & 1248       & 61632      & 1.13e-04   & 3.09       & 2.37e-06   & 4.33       & 1.12e-04   & 3.47         \\ 
      &  & 4800       & 247680     & 1.46e-05   & 2.95       & 1.54e-07   & 3.94       & 1.18e-05   & 3.24         \\ 
      &  & 18816      & 993024     & 1.81e-06   & 3.01       & 9.91e-09   & 3.96       & 1.23e-06   & 3.27         \\ 
      &  & 74496      & 3976704    & 2.26e-07   & 3.00       & 7.49e-10   & 3.73       & 1.21e-07   & 3.34         \\ 

      \midrule
      \multirow{5}{*}{\textbf{HHO-HHO}} &
      \multirow{5}{*}{\textbf{2}} &

      360        & 18792      & 9.60e-04   & --         & 4.76e-05   & --         & 7.92e-04   & --           \\ 
      & & 1296       & 73872      & 1.13e-04   & 3.09       & 2.37e-06   & 4.33       & 6.89e-05   & 3.52         \\ 
      & & 4896       & 292896     & 1.46e-05   & 2.95       & 1.54e-07   & 3.95       & 6.92e-06   & 3.32         \\ 
      & & 19008      & 1166400    & 1.81e-06   & 3.01       & 9.87e-09   & 3.96       & 6.94e-07   & 3.32         \\ 
      & & 74880      & 4655232    & 2.25e-07   & 3.01       & 7.43e-10   & 3.73       & 7.12e-08   & 3.29         \\ 

      \midrule
      \multirow{5}{*}{\textbf{HHO-DG}} & 
      \multirow{5}{*}{\textbf{3}} & 

      480        & 31488      & 1.07e-04   & --         & 4.34e-06   & --         & 1.62e-04   & --           \\ 
      & & 1792       & 128128     & 6.81e-06   & 3.97       & 1.03e-07   & 5.40       & 7.27e-06   & 4.48         \\ 
      & & 6912       & 516864     & 3.84e-07   & 4.15       & 2.90e-09   & 5.15       & 3.46e-07   & 4.39         \\ 
      & & 27136      & 2076160    & 2.42e-08   & 3.99       & 8.87e-11   & 5.03       & 1.84e-08   & 4.23         \\ 
      & & 107520     & 8322048    & 1.49e-09   & 4.02       & 2.69e-12   & 5.04       & 8.68e-10   & 4.40         \\ 

      \midrule
      \multirow{5}{*}{\textbf{HHO-HHO}} &
      \multirow{5}{*}{\textbf{3}} &

      480        & 33408      & 1.07e-04   & --         & 4.34e-06   & --         & 1.22e-04   & --           \\ 
      & & 1728       & 131328     & 6.82e-06   & 3.97       & 1.03e-07   & 5.40       & 4.51e-06   & 4.75         \\ 
      & & 6528       & 520704     & 3.84e-07   & 4.15       & 2.90e-09   & 5.15       & 2.05e-07   & 4.46         \\ 
      & & 25344      & 2073600    & 2.43e-08   & 3.98       & 8.83e-11   & 5.04       & 1.08e-08   & 4.25         \\ 
      & & 99840      & 8275968    & 1.50e-09   & 4.02       & 2.68e-12   & 5.04       & 5.21e-10   & 4.38         \\ 
      \bottomrule
    \end{tabular}
  \end{small}
  \caption{\label{tab:smalldiff} Convergence rates for HHO-HHO and HHO-DG discretisations based on
    manufactured solutions of the Biot problem, see text for details. Relevant parameters are $\mu = \lambda = 1$ and  
    $\diff = \kappa \Id$, with $\kappa = 10^{-6}$.}
\end{table}


\bibliographystyle{plain}
\bibliography{abbho}

\end{document}